\documentclass[11pt,reqno]{amsart}
\usepackage{amsmath,amssymb,amsfonts,array,amsthm,mathrsfs}
\usepackage{cases}
\usepackage{inputenc, enumerate, color}
\usepackage{graphicx}
\usepackage{fullpage}
\usepackage{times}
\usepackage{enumerate}
\usepackage{txfonts}
\usepackage{verbatim}
\usepackage[table,dvipsnames,usenames]{xcolor}
\usepackage[colorlinks=true,urlcolor=Black,citecolor=Black,linkcolor=Black]{hyperref}
\usepackage{hyperref}
\usepackage{url}
\usepackage{calc}
\usepackage{multirow}
\usepackage{caption}
\usepackage{relsize}

\makeatletter
\@namedef{subjclassname@2010}{%
  \textup{2010} Mathematics Subject Classification}
\makeatother

\newtheorem*{rep@theorem}{\rep@title}
\newcommand{\newreptheorem}[2]{%
\newenvironment{rep#1}[1]{%
 \def\rep@title{#2 \ref{##1}}%
 \begin{rep@theorem}}%
 {\end{rep@theorem}}}
\makeatother

\usepackage{amsrefs}
\newenvironment{rezabib}
  {\bibdiv\biblist\setupbib}
  {\endbiblist\endbibdiv}

\def\setupbib{\catcode`@=\active}
\begingroup\lccode`~=`@
  \lowercase{\endgroup\def~}#1#{\gatherkey{#1}}
\def\gatherkey#1#2{\gatherkeyaux{#1}#2\gatherkeyaux}
\def\gatherkeyaux#1#2,#3\gatherkeyaux{\bib{#2}{#1}{#3}}

\newtheorem{thm}{Theorem}[section]
\newtheorem{thrm}[thm]{Theorem}
\newtheorem{prop}[thm]{Proposition}
\newtheorem{remark}[thm]{Remark}
\newtheorem{remarks}[thm]{Remarks}
\newtheorem{lem}[thm]{Lemma}

\newtheorem{cor}[thm]{Corollary}
\newtheorem{conj}[thm]{Conjecture}
\newtheorem{notation}[thm]{Notation}

\newcommand{\QQ}{\mathbb{Q}}
\newcommand{\CC}{\mathbb{C}}

\newcommand{\ZZ}{\mathbb{Z}}

\DeclareMathOperator{\tr}{tr}

\DeclareMathOperator{\Gal}{Gal}
\DeclareMathOperator{\GL}{GL}

\newcommand{\mt}{m(t,u;\ell^k)}
\newcommand{\mtp}{m(t,p;\ell^k)}
\newcommand{\mtone}{m(t_1,u;\ell^k)}
\newcommand{\mttwo}{m(t_2,u;\ell^k)}

\newcommand{\Sts}{S(t_1,t_2;\ell^k)}
\newcommand{\Sttwo}{S(t_1,t_2;2^k)}
\newcommand{\Stt}{S(t;\ell^k)}
\newcommand{\Sk}{\mathcal{S}_k}
\newcommand{\Stwo}{S(t;2^k)}

\newcommand{\zstar}{(\mathbb{Z}/\ell^k\ZZ )^*}
\newcommand{\zstartwo}{(\mathbb{Z}/2^k\ZZ )^*}
\newcommand{\mee}{m_{E_1,E_2}}

\newcommand{\eif}{\emph{\text{if }}}

\newcommand{\eodd}{\emph{\text{ odd}}}
\newcommand{\eand}{\emph{\text{ and }}}

\makeatletter
\def\imod#1{\allowbreak\mkern7mu({\operator@font mod}\,\,#1)}
\makeatother

\begin{document}

\title{On the Lang-Trotter Conjecture for two elliptic curves}
\author{Amir Akbary and James Parks}

\thanks{Research of the first author is partially supported by NSERC. Research of the second author is partially supported by a PIMS postdoctoral fellowship.}

\date{\today}



\address{Department of Mathematics and Computer Science, University of Lethbridge,\newline
	\rule[0ex]{0ex}{0ex}\hspace{8pt}  4401 University Drive, Lethbridge, AB, T1K 3M4, Canada}
\email{amir.akbary@uleth.ca}
\address{Department of Mathematics and Computer Science, University of Lethbridge,\newline
	\rule[0ex]{0ex}{0ex}\hspace{8pt}  4401 University Drive, Lethbridge, AB, T1K 3M4, Canada\newline
	\rule[0ex]{0ex}{0ex}\hspace{8pt} \textit{Present address}: Department of Mathematics, KTH Royal Institute of Technology, \newline
	\rule[0ex]{0ex}{0ex}\hspace{8pt}  Lindstedtsv\"agen 25, SE-100 44 Stockholm, Sweden}
\email{jparks@kth.se}

\begin{abstract} 
Following Lang and Trotter we  describe a probabilistic model that predicts the distribution of primes $p$ with given  Frobenius traces at $p$ for two fixed elliptic curves over $\mathbb{Q}$.
In addition, we propose explicit Euler product representations for the constant in the predicted asymptotic formula and describe in detail the universal component of this constant. A new feature is that in some cases the $\ell$-adic limits determining the $\ell$-factors of the universal constant, unlike the Lang-Trotter conjecture for a single elliptic curve, do not stabilize. We also prove the conjecture on average over a family of elliptic curves,  which extends the main results of \cite{FM2} and \cite{ADJ},  following the work of David, Koukoulopoulos, and Smith \cite{DKS}.
\end{abstract}

\maketitle

\section{Introduction}
Let $E$ be an elliptic curve defined over $\QQ$. We let $a_p(E)$ denote the trace of the Frobenius endomorphism of $E$ at a prime $p$ of good reduction. Remarkable progress in recent years led to the proof of the Sato-Tate Conjecture by Clozel, Harris, Shepherd-Barron, and Taylor (see \cite{HSBT} and references there), which describes the distribution of $a_p(E)$ in the Hasse interval $(-2\sqrt{p}, 2\sqrt{p})$. However, the related Lang-Trotter Conjecture \cite[p. 33]{LT} from 1976 is still a wide open problem.  

We denote the algebraic closure of $\mathbb{Q}$ by $\overline{\QQ}$ and for a prime $\ell$ we let  $\GL_2(\ZZ_\ell)$ be the group of invertible two by two matrices with coordinates in the set of $\ell$-adic integers $\ZZ_\ell$.  Let $$\rho_E: \Gal(\overline{\QQ}/\QQ) \rightarrow \prod_{\ell}\GL_2(\ZZ_\ell)$$ be the Galois representation associated to $E$ and $\ell$-torsion points. For an integer $m>1$, let $$\phi_m :  \prod_\ell \GL_2(\ZZ_\ell) \rightarrow \GL_2(\ZZ/m\ZZ) $$ be the natural projection map.  We denote the image of $\rho_{E, m}=\phi_m\circ\rho_E$ by $G_E(m)$. Also we write ${G_E(m)}_t$ for the set of elements of $G_E(m)$ with trace $t$. We now state the Lang-Trotter Conjecture.

\begin{conj}(\textbf{Lang-Trotter}) \label{conj: langtrot sing} Let $E$ be an elliptic curve defined over $\QQ$ of conductor $N_E$ and  without complex multiplication. Let $t$ be an integer. Then 
	\begin{equation*}
	\pi_{E,t}(x) :=\#\{p \leq x;~ p\nmid N_E~{\rm and}~ a_p(E)=t\}\sim c_{E,t} \frac{\sqrt{x}}{\log x},
	\end{equation*}
as $x\rightarrow \infty$, where $$c_{E,t}= \frac{2}{\pi} \lim_{m \widetilde{\rightarrow} \infty}  \frac{m\cdot | {G_E(m)}_t|}{| G_E(m)|}.$$
\end{conj}
Here we use the notion $\widetilde{\rightarrow}$ as introduced in \cite[Section 2.3]{CDSS}.  More precisely, for a sequence $(s_n)$, we set 
$$\lim_{m\widetilde{\rightarrow}\infty} s_m:= \lim_{n\rightarrow \infty}s_{m_n}~~{\rm with}~~ m_n:=\prod_{\ell\leq n} \ell^n.$$ 

We note that a similar conjecture has also been proposed for the case that $E$ has complex multiplication (CM) and $t\neq 0$ (see \cite[Conjecture 1]{BJ}). However, the analysis of the constant is different in the CM and non-CM cases. In this paper, we are only interested in the non-CM case.

The constant $c_{E, t}$ can be zero in certain cases. For example for $E: y^2=(x-1)(x-2)(x-3)$ and $p>2$, one can show that $a_p(E)$ is even (see \cite[p. 420]{K}). Thus, by considering $a_p(E)$ as the trace of the Frobenius at $p$ in the division field extension $\mathbb{Q}(E[m])/\mathbb{Q}$ and applying the Chebotarev density theorem we conclude that $|G_E(m)_t|=0$, for odd $t$, as $m\widetilde{\rightarrow}\infty$. Therefore,  $c_{E, t}=0$ for odd $t$. 

Lang and Trotter expressed the constant $c_{E, t}$ as a product of a non-negative rational number $r_{E, t}$, depending on $E$ and $t$, and a positive \emph{universal constant} $c_t$, depending only on $t$. Moreover, in \cite[Theorem 4.2]{LT}, they provide explicit expressions for $r_{E, t}$ and $c_t$. 
A celebrated theorem of Serre \cite{S} states that,  for a non-CM elliptic curve $E$, the image of $\rho_E$ is open in $\prod_{\ell}\GL_2(\ZZ_\ell)$. Therefore there exists a positive integer $m$ such that $\rho_E(\Gal(\overline{\QQ}/\QQ))=\phi_m^{-1} (G_E(m))$. Let $m_E$ be the least such $m$. Then for $m=m_1m_2$ with $m_1 \mid m_E^\infty$ (i.e. the prime divisors of $m_1$ are among the prime divisors of $m_E$) and $(m_2, m_E)=1$, we obtain
$$G_E(m)\simeq G_E(m_1)\times G_E(m_2)=G_E(m_1) \times \GL_2(\ZZ/m_2\ZZ).$$ 
Using this fact, we may then write $c_{E,t}=r_{E,t}\cdot c_t$, where
$$r_{E, t}= \lim_{k\rightarrow \infty}\frac{{\rm rad}(m_E)^k\cdot |{G_E({\rm rad}(m_E)^k)}_t|}{| G_E({\rm rad}(m_E)^k)|} \prod_{\ell\mid m_E} \left ( \lim_{k\rightarrow \infty } \frac{\ell^k\cdot|{\GL_2(\ZZ/\ell^k\ZZ)}_t|}{| \GL_2(\ZZ/\ell^k\ZZ)|}\right )^{-1},$$ 
and 
\begin{equation}
\label{uc}
c_t=\frac{2}{\pi}\prod_{\ell}   \left ( \lim_{k\rightarrow \infty } \frac{\ell^k \cdot |{\GL_2(\ZZ/\ell^k\ZZ)}_t|}{| \GL_2(\ZZ/\ell^k\ZZ)|}\right ).\end{equation}
Here ${\rm rad}(m_E)$ denotes the product of prime divisors of $m_E$. In \cite[p. 34, Lemma 2]{LT} it is shown that 
$$\lim_{k\rightarrow \infty}\frac{{\rm rad}(m_E)^k\cdot |{G_E({\rm rad}(m_E)^k)}_t|}{| G_E({\rm rad}(m_E)^k)|}= \frac{m_E\cdot |{G_E(m_E)}_t|}{| G_E(m_E)|}.$$
Also from \cite[Theorem 4.1]{LT} we have that
$$c_t=\frac{2}{\pi}\prod_{\ell}    \frac{\ell \cdot |{\GL_2(\ZZ/\ell \ZZ)}_t|}{| \GL_2(\ZZ/\ell\ZZ)|}=\frac{2}{\pi}  \prod_{\ell\nmid t} \frac{\ell^3-\ell^2-\ell}{(\ell^2-1)(\ell-1)} \prod_{\ell \mid t} \frac{\ell^2}{\ell^2-1}.$$

The Lang-Trotter Conjecture has been studied extensively in the literature. 
The best known unconditional upper bound for $\pi_{E, t}(x)$ for $t=0$  is $ x^\frac34$,  obtained by  Elkies \cite{Elk} and Ram Murty, and is
$ x(\log\log x)^2/(\log x)^2$ for $t\neq 0$ (see \cite[Theorem 5.1]{KM} and \cite[Theorem 1.4]{TZ}). Under GRH, Zywina \cite{Z} has recently obtained an upper bound for  $\pi_{E, t}(x)$ of  size $x^\frac45/(\log x)^\frac35$ for $t\neq 0$, and  an upper bound of size $ x^\frac34/(\log x)^\frac12$ for $t= 0$. 
The Lang-Trotter Conjecture was first shown to hold on average over a family of elliptic curves in the case $t=0$ by Fouvry and Ram Murty \cite{FM}. This result was then extended to the case of non-zero integers by David and Pappalardi \cite{DP}. 

In \cite[Remark 2, p. 37]{LT} it is mentioned that, by employing a probabilistic model, one can state an analogous conjecture for two elliptic curves. To our knowledge an exact statement of this conjecture with a conjectural constant has not appeared in the literature. In fact the previous related work \cite{FM2} and \cite{ADJ} consider the Frobenius distribution for two elliptic curves on average over a family of pairs of elliptic curves. In Section \ref{section: prob}, inspired by the model  developed by Lang and Trotter in \cite{LT},  we employ a probabilistic model to propose an explicit conjecture on the distribution of primes with two given traces for two fixed elliptic curves. Here we describe the conjecture.

Let $E_1$ and $E_2$ be two elliptic curves defined over $\QQ$ without complex multiplication that are not isogenous over $\overline{\QQ}$. Let $$\rho_{E_1, E_2}: \Gal(\overline{\QQ}/\QQ) \rightarrow \prod_{\ell}\GL_2(\ZZ_\ell) \times \GL_2(\ZZ_\ell) $$ be the Galois representation associated to $E_1$ and $E_2$ and their $\ell$-torsion points. For an integer $m>1$, let $$\phi_m :  \prod_\ell \GL_2(\ZZ_\ell) \times \GL_2(\ZZ_\ell)\rightarrow \GL_2(\ZZ/m\ZZ) \times \GL_2(\ZZ/m\ZZ)  $$ be the natural projection map. We denote the image of $\rho_{E_1, E_2, m}=\phi_m\circ\rho_{E_1, E_2}$ by $G_{E_1, E_2}(m)$. As a consequence of the Weil paring, we know that 
for $\sigma\in \Gal(\overline{\QQ}/\QQ)$ and an $m$-th root of unity $\zeta_m$, we have
$$\sigma(\zeta_m)=\zeta_m^{{\rm det}(\rho_{E_1,m})} = \zeta_m^{{\rm det}(\rho_{E_2, m})}$$
(see \cite[III.8]{Sil}). Thus,  $G_{E_1, E_2}(m)\subseteq \Delta(\ZZ/m\ZZ),$
where $$\Delta(\ZZ/m\ZZ):=\{ (g_1, g_2)\in \GL_2(\ZZ/m\ZZ)\times \GL_2(\ZZ/m\ZZ); ~\det( g_1)=\det (g_2)\}.$$
Furthermore, we  set 
$$G_{E_1, E_2}(m)_{t_1, t_2}:=\{(g_1, g_2)\in G_{E_1, E_2}(m);~ \tr(g_1)=t_1\text{ and }\tr(g_2)=t_2\}.$$

By developing a model similar to \cite[p. 29--32]{LT}, we propose the following conjecture.
\begin{conj} \label{conj: langtrot 2} 
Let $E_1$ and $E_2$ be two non $\overline{\QQ}$-isogenous elliptic curves defined over $\QQ$ with conductors $N_1$ and $N_2$ respectively and without complex multiplication. Then, for fixed integers $t_1$ and $t_2$, we have that
	\begin{equation*}
	\pi_{E_1,E_2, t_1,t_2}(x) :=\#\{p \leq x;~p\nmid N_1N_2,~ a_p(E_1)=t_1, \eand a_p(E_2)=t_2 \}\sim c_{E_1,E_2, t_1,t_2} \log\log x,
	\end{equation*}
as $x\rightarrow\infty$,  where $$c_{E_1, E_2, t_1, t_2}= \frac{1}{\pi^2} \lim_{m \widetilde{\rightarrow} \infty}  \frac{m^2\cdot | {G_{E_1, E_2}(m)}_{t_1, t_2}|}{| G_{E_1, E_2}(m)|}.$$\end{conj}

\begin{remarks}
{\em (i) If $E_1$ and $E_2$ are $\overline{\QQ}$-isogenous, then for $p\nmid N_1 N_2$, we have $a_p(E_1)=a_p(E_2)$, thus
$$\pi_{E_1, E_2, t_1, t_2}(x)=\pi_{E_1, t_1}(x)+O_{E_2}(1)= \pi_{E_2, t_2}(x)+O_{E_1}(1).$$
Thus Conjecture \ref{conj: langtrot sing} predicts the behaviour of $\pi_{E_1, E_2, t_1, t_2}(x)$ in this case. }

{\em (ii) The constant $c_{E_1, E_2, t_1, t_2}$ can be zero. For example, for $i=1,2$, let $E_i: y^2=(x-x_{i,1})(x-x_{i,2})(x-x_{i,3})$, where $x_{i,1}$, $x_{i,2}$, and $x_{i,3}$ are distinct integers. Then $c_{E_1, E_2, t_1, t_2}=0$ if $t_1$ and $t_2$ are odd.} 

{\em (iii) One should be able to formulate a similar conjecture if $E_i$ (for $i=1$ or $2$) has complex multiplication, as long as $t_i\neq 0$. It is known that the Lang-Trotter conjecture for CM curves in certain cases is equivalent to the conjecture on the distribution of  primes generated by certain quadratic polynomials. For example for the curve $E_1: y^2=x^3-x$, we have, for $p\neq 2$,  that if $a_p(E_1)=2$, then $p=n^2+1$ for some integer $n$ (see \cite[p. 307, Theorem 5]{IR}). Similarly for the curve $E_2: y^2=x^3+1$, we have, for $p\neq 2, 3$,  that if $a_p(E_2)=1$, then $p=3n^2+3n+1$ for some integer $n$ (see \cite[p. 305, Theorem 4]{IR}).  Thus the corresponding conjecture for $\pi_{E_1, E_2, 2, 1}(x)$ is related to the distribution of primes generated simultaneously by the polynomials $n^2+1$ and $3n^2+3n+1$.
}

{\em (iv) Due to the slow growth of the double logarithm function, unlike the Lang-Trotter conjecture for one elliptic curve, obtaining substantial experimental evidence for this conjecture is outside the realm of current computational power. For example, for two pairs of elliptic curves studied in \cite{LT}, Lang and Trotter report only one coincidence of supersingular primes among the first $5000$ primes (see \cite[p. 38]{LT}).}

\end{remarks}

The conjectural constant $c_{E_1, E_2, t_1, t_2}$ in Conjecture \ref{conj: langtrot 2} is the focus of this paper. We first observe that, by a theorem of Serre \cite[Th\'eor\`eme 6]{S} and by the work of Faltings  on the Tate Conjecture \cite{F}, there exists an analogue of Serre's open image theorem for two non-CM elliptic curves $E_1$ and $E_2$ that are not $\overline{\mathbb{Q}}$-isogenous (see also \cite[p. 3383]{J2}). Thus following the steps described above for the constant in the Lang-Trotter conjecture, we can establish the existence of a positive integer $m_{E_1, E_2}$ such that 
$c_{E_1, E_2, t_1, t_2}=r_{E_1, E_2, t_1, t_2}\cdot c_{t_1, t_2}$, where
$$r_{E_1, E_2, t_1, t_2}=\lim_{k\rightarrow \infty}\frac{{\rm rad}(m_{E_1, E_2})^{2k}\cdot |{G_{E_1, E_2}({\rm rad}(m_{E_1, E_2})^k)}_{t_1, t_2}|}{| G_{E_1, E_2}({\rm rad}(m_{E_1, E_2})^k)|}\prod_{\ell\mid \mee} \left ( \lim_{k\rightarrow \infty } \frac{\ell^{2k} \cdot |{\Delta(\ZZ/\ell^k\ZZ)}_{t_1, t_2}|}{|\Delta(\ZZ/\ell^k\ZZ)|}\right )^{-1},$$ 
and 
\begin{equation}
\label{defn: ctonettwo delta}
c_{t_1, t_2}=\frac{1}{\pi^2}\prod_{\ell}   \left ( \lim_{k\rightarrow \infty } \frac{\ell^{2k} \cdot |{\Delta(\ZZ/\ell^k\ZZ)}_{t_1, t_2}|}{| \Delta(\ZZ/\ell^k\ZZ)|}\right ).
\end{equation}
Here, $\Delta(\ZZ/\ell^k\ZZ)_{t_1, t_2}$ is the collection of $(g_1, g_2)\in \Delta(\ZZ/\ell^k\ZZ)$ with ${\rm tr}(g_1)=t_1$ and ${\rm tr}(g_2)=t_2$.

In this paper we propose explicit representations as rational functions of $\ell$ for the universal constant $c_{t_1, t_2}$ for different values of $t_1$ and $t_2$. We point out that, unlike the computation of $c_t$ for a single elliptic curve, determining explicit formulas for $c_{t_1, t_2}$ is an intricate problem that involves combinatorial computations in the ring of matrices with entries in $\ZZ/\ell^k\ZZ$. We denote the set of two by two matrices with entries in this ring by ${\rm M}_2(\ZZ/\ell^k \ZZ)$ and  for $t, u\in \mathbb{Z}/\ell^k\ZZ$ we define
\begin{equation}
\label{def: mtplk}
\mt:= \#\{A\in {\rm M}_2(\ZZ/\ell^k \ZZ);~ \tr(A)=t~ \text{ and } \det(A)=u\}.
\end{equation}

Our first result gives a representation of the universal constant in terms of specific matrix counts.
\begin{thrm}\label{thrm: ctonetwo}
For $t_1, t_2\in \ZZ$, set
\begin{equation}
\label{def-st}
\Sts:=\sum_{u\in \zstar } \mtone \mttwo.
\end{equation}
Then 
$$\lim_{k\rightarrow \infty } \frac{\Sts}{\ell^{5k-5}}$$
exists and  the constant $c_{t_1, t_2}$ is given by the following convergent  Euler product:
$$c_{t_1, t_2}=\frac{1}{\pi^2}\prod_{\ell}   \left (\frac{1}{(\ell-1)^3(\ell+1)^2} \lim_{k\rightarrow \infty } \frac{\Sts}{\ell^{5k-5}}\right ).$$\end{thrm}  

We set the following notational conventions that will be used throughout the paper. For a fixed prime $\ell$ and integers $t_1, t_2$, we define 
\begin{equation}
\label{defnofsk}
\Sk:=\frac{\Sts}{\ell^{5k-5}}=\frac{1}{\ell^{5k-5}}\sum_{u\in \zstar } \mtone \mttwo.
\end{equation}
We say that $\Sk$ \textit{stabilizes at $k_0$} if for all $k\geq k_0$ we have that $\Sk=\mathcal{S}_{k_0}$. In addition, if $t_1=t_2=t$, we define $\Stt:=S(t,t; \ell^k)$. Theorem \ref{thrm: ch-thrm-five} provides a formula for $\mt$ where the only dependence on $t$ comes from $D(t,u):=t^2-4u$. Since $D(t,u)=D(-t,u)$ we have that $\mt=m(-t,u;\ell^k)$ and hence 
\begin{equation}
\label{t-minus-t}
S(\pm t_1, \pm t_2;\ell^k)= S(t_1,t_2;\ell^k).
\end{equation}

The following theorem establishes the value of $\lim_{k\rightarrow \infty} \Sk$ in the case $t_1=\pm t_2$. 

\begin{thrm}\label{samet}
	Let $\ell$ be a prime and $t$ be an integer. Then for $k\geq 3$ we have that 
	\begin{equation*}
	\frac{\Stt}{\ell^{5k-5}}=\begin{cases}
	\ell^2(\ell^2+1)(\ell-1)&\eif\ell \mid t\eand \ell>2, \\
	 \frac{\ell^2(\ell^4-2\ell^2-3\ell-1)}{\ell+1} -\frac{\ell^4}{\ell^{2k}(\ell+1)}&\eif \ell \nmid 2t, \\ 
	4 &\eif \ell=2 \eand 2\nmid t,\\
	\frac{35}{2} &\eif \ell=2 \eand 4\mid t,\\
	\frac{103}{6}-\frac{32}{3\cdot 2^{2k}} &\eif \ell=2, 2\mid t, \eand 4\nmid t.
	\end{cases}
	\end{equation*}
\end{thrm}
We have that $\mathcal{S}_k$ stabilizes at $3$ if $\ell=2$ and $t$ is a multiple of $4$. If $\ell \mid t$ for an odd prime $\ell$, or $\ell=2$ and $t$ is odd, then, in Section \ref{section: ltcomps}, we show that $\mathcal{S}_k$ also stabilizes at $1$. 

The following result, which gives exact values of  $c_{t, t}$ and $c_{t, -t}$, is an immediate consequence of Theorem \ref{thrm: ctonetwo}, Theorem \ref{samet}, and \eqref{t-minus-t}. 

\begin{cor}\label{cor: ct}
	For an integer $t$, we have
	
	\begin{equation}\label{ctequals}
	c_{t,t} = c_{t, -t}= \frac{1}{\pi^2}\prod_{\substack{\ell > 2 \\ \ell \nmid t}}\frac{\ell^2(\ell^4-2\ell^2-3\ell-1)}{(\ell^2-1)^3}\prod_{\substack{\ell >2\\ \ell \mid t}}\frac{\ell^2(\ell^2+1)}{(\ell^2-1)^2}\cdot\begin{cases} \frac{4}{9} &\eif 2\nmid t, \\ \frac{35}{18} &\eif 4\mid t, \\ \frac{103}{54} &\eif 2\mid t, 4\nmid t.\end{cases}
	\end{equation}
	
\end{cor}

\begin{remarks}\label{unstable}
{\em For $t=0$, the value $c_{0, 0}=35/96$ is rational and is the same value obtained by Fouvry and Murty in \cite{FM2} for the average Lang-Trotter conjecture for two elliptic curves over a family of elliptic curves. For $t\neq 0$, we can write 
$$c_{t, t}=c_{t, -t}=q_t \prod_{\ell} \frac{\ell^{4}-2\ell^{2}-3\ell-1}{(\ell^2-1)^2},$$  
where $q_t$ is a certain rational number depending on $t$  that can be explicitly written and the value of the Euler product is approximately  0.08789878383\ldots.}	
\end{remarks}

{The presence of cases dependent on $k$  in Theorem \ref{samet} is a new feature of the constant for the Lang-Trotter Conjecture for two elliptic curves. In all other examples in the literature for problems of this kind the analogous Euler factors in \eqref{defn: ctonettwo delta} always stabilize.} 	
{The question of  stabilization of the Euler factors  of $c_{t_1, t_2}$ is  intimately related to the smoothness of the $\ell$-adic analytic manifold   
$$\Delta(\ZZ_\ell)_{t_1, t_2}:=\{ (g_1, g_2)\in {\rm GL}_2 (\ZZ_\ell) \times  {\rm GL}_2 (\ZZ_\ell) ;~{\rm tr}(g_1)=t_1,~{\rm tr}(g_2)=t_2,~{\rm and}~{\rm det}(g_1)={\rm det}(g_2)\},$$
where $t_1, t_2\in \ZZ_\ell$. In fact the $\ell$-th factor of $c_{t, t}$ is an scaled multiple of the volume of  $\Delta(\ZZ_\ell)_{t, t}$.  More precisely, from a theorem of Oesterl\'{e} \cite[Th\'eor\`eme 2]{O} and \eqref{deltats}, we know that 
$${\rm Vol}(\Delta(\ZZ_\ell)_{t_1, t_2})=\lim_{k\rightarrow \infty} \frac{\Delta(\ZZ/\ell^k \ZZ)_{t_1, t_2}}{\ell^{5k}}=\lim_{k\rightarrow \infty} \frac{S(t_1, t_2; \ell^k)}{\ell^{5k}},$$
where the volume is with respect to a certain measure defined in \cite[p. 326]{O}.
Then Theorem \ref{samet} can be re-written as 
$${\rm Vol}(\Delta(\ZZ_\ell)_{t, t})=
\begin{cases}
	\frac{(\ell^2+1)(\ell-1)}{\ell^3}&\text{ if }\ell \mid t\text{ and } \ell>2, \\
	 \frac{\ell^4-2\ell^2-3\ell-1}{\ell^3(\ell+1)} 
	 &\text{ if } \ell \nmid 2t, \\ 
	\frac{1}{ 8}&\text{ if } \ell=2 \text{ and }2\nmid t,\\
	\frac{35}{64} &\text{ if } \ell=2 \text{ and } 4\mid t,\\
	\frac{103}{192}
	&\text{ if } \ell=2, 2\mid t \text{ and }4\nmid t.
	\end{cases}
$$

In view of a theorem of Serre \cite[Th\'eor\`eme 9]{S2}, if $\Delta(\ZZ_\ell)_{t_1, t_2}$ is a smooth $\ell$-adic analytic manifold, then $S(t_1, t_2; \ell^k)/\ell^{5k}$ is stable. A computation involving the Jacobian of equations defining $\Delta(\ZZ_\ell)_{t_1, t_2}$ reveals that  $\Delta(\ZZ_\ell)_{t_1, t_2}$  is smooth for $t_1\neq \pm t_2$. Thus, we conclude that   $S(t_1, t_2; \ell^k)/\ell^{5k}$ is stable for $t_1\neq \pm t_2$.      }

The proof of Theorem  \ref{samet} is done by employing formulas developed in \cite{G} and \cite{CH} for the function $\mt$ in \eqref{def: mtplk} on a case by case analysis. The proof, although straightforward, is tedious and breaks down into  thirteen separate cases (four when $\ell$ is odd and nine when $\ell=2$). One may consider a similar approach in studying $S(t_1, t_2; \ell^k)$ in the case $t_1\neq \pm t_2$. However, there are many more cases involved which makes the case by case analysis much more complicated. In certain cases, we have successfully applied this approach in the evaluation of the $\ell$-th factor 
$$c_\ell(t_1, t_2)= \frac{1}{(\ell-1)^3 (\ell+1)^2}\lim_{k\rightarrow \infty } \frac{\Sts}{\ell^{5k-5}}$$
in the Euler product of the constant 
$$c_{t_1, t_2}=\frac{1}{\pi^2} \prod_\ell c_\ell (t_1, t_2).$$

\begin{prop}
\label{proved-cases}
Let $t_1$, $t_2$, where $t_1\neq \pm t_2$, be two integers and let $\ell$ be a prime. For $\ell>2$, if $\ell\mid t_1t_2$ then we have 
\begin{equation*}
c_\ell(t_1,t_2)=
	\frac{\ell^2\bigg(\ell^3-\ell^2+\Big(1-2 \left(\tfrac{(t_1,t_2)}{\ell}\right)^2\Big)\ell-1\bigg)}{(\ell-1)^3(\ell+1)^2},
	\end{equation*}
	where $(\tfrac{\cdot}{\ell})$ is the Legendre symbol. For $\ell=2$, we have
	\begin{align*}
c_2(t_1,t_2)&=\begin{cases}
\tfrac 49 & \eif 2 \nmid t_1t_2, \\ 
\tfrac 8 9 & \eif 2\mid t_1t_2\eand 2 \nmid (t_1,t_2),\\
\tfrac{33}{18} & \eif 4\mid (t_1,t_2) \eand t_1^2\not\equiv t_2^2 \imod{16},\\
\tfrac{35}{18} & \eif 4\mid (t_1,t_2) \eand t_1^2\equiv t_2^2 \imod{16}.
\end{cases}
\end{align*}

\end{prop}

For the remaining cases, not covered in the above proposition, we followed a computational approach.
We implemented an algorithm in SAGE to compute $\Sts$ for various small primes $\ell$, integers $t_1,t_2$, and positive integers $k$ using Theorem \ref{thrm: ch-thrm-five}   and then used rational interpolation approximation in MAPLE to represent $S(t_1, t_2, \ell^k)$ as a rational  function of $\ell$ for various values of $t_1$ and $t_2$. In Section \ref{section: comp} we provide computational evidence for our conjectured values of $S(t_1, t_2; \ell^k)$.  
We state our findings as a conjecture.

We denote the $\ell$-adic valuation of an integer $m$ by $\nu_\ell(m)$.

\begin{conj} \label{conj: constant-distinct-ts} 	
Let $t_1$, $t_2$, where $t_1\neq \pm t_2$, be two integers and let $\ell$ be a prime. Set $$\alpha=\alpha(t_1,t_2,\ell):=\max\{\nu_\ell(t_1+t_2),\nu_\ell(t_1-t_2)\}.$$ For $\ell>2$, if  $\ell\nmid t_1 t_2$, then we have
	\begin{align*}
c_\ell(t_1,t_2)&= \begin{cases} \tfrac{\ell^2(\ell^3-\ell^2-2\ell-2)}{(\ell-1)^3(\ell+1)^2}&\eif \alpha=0, \\
\tfrac{\ell^2(\ell^4-2\ell^2-3\ell-1)-\ell^{2-2\alpha}(\ell^2+\ell+1)}{(\ell-1)^3(\ell+1)^3} & \eif \alpha\geq 1.\\
\end{cases}
\end{align*}
For $\ell=2$, we have
\begin{align*}
c_2(t_1,t_2)&=\begin{cases}
\tfrac{15}{9} & \eif  2\mid (t_1,t_2), 4\nmid (t_1,t_2), \eand \alpha=1,\\
\tfrac{103}{54} - \tfrac{7}{27\cdot2^{2\alpha-3}}& \eif 2\mid(t_1,t_2), 4\nmid (t_1,t_2), \eand \alpha\geq 3.\\
\end{cases}
\end{align*}
\end{conj}

\begin{remarks}
{\em (i) A version of this conjecture in terms of the quantity
${\mathcal{S}}_k={\Sts}/{\ell^{5k-5}}$ is given in Conjecture \ref{differentt}. We note that for $t_1\neq \pm t_2$ the quotient $S(t_1, t_2; \ell^k)/\ell^{5k-5}$ stabilizes at $k=\alpha+1$ as $k
\rightarrow \infty$. }

{\em (ii) Observe that for $t_1=\pm t_2$ we have that $\alpha=\infty$. In this case Conjecture \ref{conj: constant-distinct-ts} together with Proposition \ref{proved-cases} imply 
Corollary \ref{cor: ct}.}
\end{remarks}

We next describe another interpretation of the $\ell$-factors in the universal constant for the Lang-Trotter conjecture for two elliptic curves. In order to do this, for an integer $t$ and primes $p$ and $\ell$, we first set

\begin{equation}
\label{equation: nuell}
f_\ell(t, p):= \lim_{k\rightarrow \infty} \frac{\ell^{k} \mtp}{\left|\rm{SL}_{2}(\ZZ/\ell^{k}\ZZ)\right|}=\lim_{k\rightarrow\infty}\frac{\mtp}{\ell^{2k-2}(\ell^2-1)},
\end{equation}
where $\mtp$ is defined in $\eqref{def: mtplk}$ and $\rm{SL}_{2}(\ZZ/\ell^{k}\ZZ)$ is the subgroup of matrices of determinant $1$ in $\rm{GL}_{2}(\ZZ/\ell^{k}\ZZ)$. Observe that for $\ell\neq p$, we have
$$f_\ell(t, p)=\lim_{k\rightarrow\infty} \frac{\ell^k\left|{\rm GL}_{2}^{(p)} (\ZZ/\ell^k\ZZ)_{t}\right|}{\left|{\rm GL}_{2}^{(p)} (\ZZ/\ell^k\ZZ)  \right|},$$
where the superscript $(p)$ means that the matrices $A\in{\rm GL}_{2}(\ZZ/\ell^k\ZZ)$ satisfy the extra condition ${\rm det}(A)=p$. Note that, for $\ell\neq p$, $f_\ell(t, p)$ is similar to the $\ell$-factor of the universal constant $c_t$ in \eqref{uc} with the imposed extra condition ${\rm det}(A)=p$.  Gekeler has shown that in \eqref{equation: nuell} the limit exists. More precisely,
\begin{equation} 
\label{fell}
f_\ell(t, p)= (1-\ell^{-2})^{-1} \cdot\begin{cases}1+\ell^{-1} &\text{ if }\left (\frac{(t^2-4p)/\ell^{2\delta}}{\ell} \right)=1, \\
1+\ell^{-1}-2\ell^{-\delta-1}&\text{ if }\left (\frac{(t^2-4p)/\ell^{2\delta}}{\ell} \right)=-1,\\
1+\ell^{-1} -(\ell+1) \ell^{-\delta-2}&\text{ if }\left (\frac{(t^2-4p)/\ell^{2\delta}}{\ell} \right)=0. \end{cases}
\end{equation} 
In \eqref{fell}, $\delta:=\delta(t, p)$ is the largest integer $i\geq 0$ such that $\ell^{2i} \mid t^2-4p$ for $\ell>2$, with the additional constraint in the case $\ell=2$ that $(t^2-4p)/2^{2i}\equiv 0$ or $1 \imod {4}$
 (see \cite[Corollary 4.6]{G}). The following proposition provides an interpretation of the $\ell$-factor of the universal constant $c_{t_1, t_2}$ as an average value of the sequence $(f_\ell(t_1, p)f_\ell(t_2, p))_p$ as $p$ varies over primes. 

\begin{prop}\label{prop: primesum}
Let $p$ be a prime and let $t_1, t_2\in \mathbb{Z}$. Let $\pi(x)$ denote the number of primes not exceeding $x$. Then  
\begin{equation}\label{sts limit}
	\lim_{x\rightarrow \infty} \frac{1}{\pi(x)} \sum_{p\leq x} f_\ell(t_1, p)f_\ell(t_2, p)=
	\frac{1}{(\ell-1)^3(\ell+1)^2} \lim_{k\rightarrow \infty} \frac{\Sts}{\ell^{5k-5}}.
\end{equation}
\end{prop}

Gekeler has also shown a remarkable relation between $f_\ell(t,p)$ and a certain class number.  Let $D$ be a negative discriminant (i.e. $D$ is a negative integer that is congruent to $0$ or $1$ modulo $4$). Let $f$ be the largest square such that $f^2\mid D$ and $D_0=D/f^2$ is congruent to $0$ or $1$ modulo $4$. In this case $D_0$ is the discriminant of the imaginary quadratic field $\mathbb{Q}(\sqrt{D_0})$ and $\mathcal{O}_{D}= \mathbb{Z}+f \mathcal{O}_{D_0}$ is an order of index $f$ in $\mathcal{O}_{D_0}$ (the ring of integers $\mathbb{Q}(\sqrt{D_0})$). Let $h(D)$ denote the \emph{class number} of $\mathcal{O}_D$. The Hurwitz-Kronecker class number of $\mathcal{O}_D$ is defined as
$$\mathcal{H}(D):=\sum_{f^\prime \mid f} h({f^\prime}^2 D_0).$$
We also define a scaled version of $\mathcal{H}(D)$. For a negative discriminant $D$, let $w(D)$ be  the size of the unit group of $\mathcal{O}(D)$ and set
$$H(D):= \sum_{f^\prime \mid f} \frac{h({f^\prime}^2 D_0)}{w({f^\prime}^2 D_0)}.$$
The following is \cite[Theorem 5.5]{G}.
\begin{thrm}[\textbf{Gekeler}]
	\label{theorem: product formula}
Let $p$ be a prime and let $t$ be an integer such that $t^2-4p<0$. Then we have 
that
\begin{equation*}
\label{H-star}
H(t^2-4p)={p} f_\infty (t, p) \prod_{\ell \text{ \rm prime }}  f_\ell (t, p), 
\end{equation*}
where $f_\ell(t,p)$ is defined in \eqref{equation: nuell} and 
\begin{align*}
f_\infty (t, p)&:= \begin{cases} \frac{1}{\pi\sqrt{p}} \sqrt{ 1-\frac{t^2}{4p}}& \emph{\text{for real }} t \emph{\text{ with }} |t|\leq 2\sqrt{p},\\ 0& \emph{\text{otherwise}}. \end{cases} \label{nu-infinite}
\end{align*}
\end{thrm}

We now consider the product of $\eqref{sts limit}$ over all primes $\ell$ to get
\begin{equation}
\label{interchange}
{\pi^2 c_{t_1, t_2}}=\prod_{\ell} \frac{1}{(\ell-1)^3 (\ell+1)^2}\lim_{k\rightarrow \infty} \frac{S(t_1, t_2; \ell^k)}{\ell^{5k-5}}=\prod_{\ell}  \lim_{x\rightarrow \infty}\frac{1}{\pi(x)} \sum_{p\leq x}f_\ell(t_1, p)f_\ell(t_2, p).
\end{equation}
If we can interchange the product and the sum in the above identity, we get that 
\begin{align*}
\sum_{p\leq x} \prod_{\ell} f_\ell(t_1, p)f_\ell(t_2, p)&\sim{\pi^2 c_{t_1, t_2}}\frac{x}{\log{x}},
\end{align*}
as $x\rightarrow\infty$.
Then by employing Theorem \ref{theorem: product formula} and the partial summation formula in   
the above sum we can conclude that
\begin{equation}
\label{identity-1}
c_{t_1, t_2}=\lim_{x\rightarrow \infty} \frac{1}{\log\log{x}}\underset{p\leq x}{\sum \nolimits^{\prime}} \frac{H(t_1^2-4 p)H(t_2^2-4 p)}{p^2}.
\end{equation}
Here ${\sum \nolimits^{\prime}} $ means that the sum is taken over primes $p>\max\{3, t_1^2/4, t_2^2/4\}$.
Thus under the assumption of the interchange of the product and the sum in \eqref{interchange} we find a new interpretation of the universal constant $c_{t_1, t_2}$ as an average value of certain class numbers. Of course, the interchange assumption in \eqref{interchange} is highly non-trivial. 

In our final result, by applying a general theorem of David, Koukoulopoulos, and Smith \cite[Theorem 4.2]{DKS}, we prove that the identity \eqref{identity-1} holds.

\begin{thrm}\label{thrm: classnum}
Let $t_1,t_2\in \ZZ$ and let $x\geq \max\{3, t_1^2/4, t_2^2/4\}$. Then
$$\underset{p\leq x}{\sum \nolimits^{\prime}} \frac{H(t_1^2-4 p)H(t_2^2-4 p)}{p^2}\sim c_{t_1, t_2} \log{\log{x}},$$
as $x\rightarrow\infty$, where $c_{t_1, t_2}$ is the constant given in Theorem \ref{thrm: ctonetwo}. 
\end{thrm}

We note that in \cite[Theorem 1.2]{ADJ} the following average asymptotic result for a family of elliptic curves is proved. 
\begin{thrm}[\textbf{Akbary-David-Juricevic}]\label{thrm: ADJ}
For $i=1, 2$, let $E_i (a_i, b_i)$ be an elliptic curve given by the Weierstrass  equation 
$$y^2=x^3+a_ix+b_i,$$
{where }$a_i, b_i\in \mathbb{Z}.$
Let $\epsilon > 0$, and let $t$ be an odd integer. For $A,B>0$ with $A,B\geq x^{1+\epsilon}$, as $x\rightarrow \infty$, we have that
$$\frac{1}{16A^2B^2}\sum_{\substack{|a_1|,|a_2|\leq A \\ |b_1|,|b_2|\leq B}}\pi_{E_1,E_2,t, t}(x)\sim  c_{t, t}\log\log x,$$
where $c_{t, t}$ is the constant given in \eqref{ctequals}.
\end{thrm}
From \cite[Formula (6)]{ADJ} we have, for $A, B\geq x^{1+\epsilon}$ with $\epsilon>0$, that
\begin{equation}
\label{equality}
\frac{1}{16A^2B^2}\sum_{\substack{|a_1|,|a_2|\leq A \\ |b_1|,|b_2|\leq B}}\pi_{E_1,E_2,t, t}(x)= \underset{p\leq x}{\sum \nolimits^{\prime}} \frac{{\left(H(t^2-4 p)\right)}^2}{p^2}+O(1).
\end{equation}
Thus Theorem \ref{thrm: classnum} establishes an extension of the result of Theorem \ref{thrm: ADJ} for even integers $t$. For $t=0$  this result was previously proved by Fouvry and Murty \cite[Theorem 1]{FM2} with better bounds on the variables $A$ and $B$ and with $c_{0, 0}=35/96$. Also, since a formula analogous to \eqref{equality} holds for distinct $t_1$ and $t_2$, as a corollary to Theorem \ref{thrm: classnum}, we conclude that for $A, B>0$, with $A, B\geq x^{1+\epsilon}$, as $x\rightarrow\infty$, we have $$\frac{1}{16A^2B^2}\sum_{\substack{|a_1|,|a_2|\leq A \\ |b_1|,|b_2|\leq B}}\pi_{E_1,E_2,t_1, t_2}(x)\sim  c_{t_1, t_2}\log\log x,$$
where $c_{t_1, t_2}$ is the constant given in Theorem \ref{thrm: ctonetwo}.

The structure of the paper is as follows. In Section \ref{section: prob}, we describe a probabilistic framework for Conjecture \ref{conj: langtrot 2}. 
More precisely, we show that applying the law of large numbers for certain random variables defined in a probabilistic model predicts the conjectured asymptotic in Conjecture \ref{conj: langtrot 2}. In Section \ref{section: theorem onethree proof}, we give a proof of Theorem \ref{thrm: ctonetwo}. In Section \ref{section: ltcomps}, we provide a detailed case by case analysis leading to proofs of Theorem \ref{samet} and Proposition \ref{proved-cases}. The proofs of Proposition  \ref{prop: primesum} and Theorem \ref{thrm: classnum} are given in Sections \ref{section: proof of prop} and \ref{section: DKS} respectively. Finally, our computational approach in Conjecture \ref{conj: constant-distinct-ts} is described in Section \ref{section: comp}. 

\begin{notation}
{\em Throughout the paper $\ell$ and $p$ denote primes. The number of primes not exceeding $x$ is denoted by $\pi(x)$. We use the usual asymptotic notation of analytic number theory in the paper.
By abuse of notation we use $t$ to denote an integer or its associated class in the ring $\ZZ/\ell^k \ZZ$. Similarly for a class $u$ in $\ZZ/\ell^k \ZZ$ we denote an integer representative of this class by $u$. The notation $(\ZZ/\ell^k\ZZ)^*$ is used for the multiplicative group of invertible elements in $\ZZ/\ell^k \ZZ$. We denote the $\ell$-adic valuation of a non-zero  integer $m$ by $\nu_\ell(m)$ and we let $\nu_\ell(0)=\infty$. 
For a ring $R$, we use the usual notations ${\rm M}_2(R)$, ${\rm GL}_2(R)$, and ${\rm SL}_2(R)$, respectively for the set of two by two matrices with entries in $R$, the group of invertible elements of ${\rm M}_2(R)$, and the subgroup of ${\rm GL}_2(R)$  consisting of two by two matrices with determinant $1$.}
\end{notation}

\section{The probabilistic model}\label{section: prob}

We follow \cite[pages 29-38]{LT} closely and build a probabilistic model to predict the density of primes $p$ with given traces $t_1$ and $t_2$ of the Frobenius endomorphism at $p$  for two fixed non $\overline{\QQ}$-isogenous non-CM elliptic curves $E_1$ and $E_2$.  Let 
\begin{equation}
\label{up}
U_p=\left((-2\sqrt{p}, 2\sqrt{p})\cap \mathbb{Z} \right)\times \left((-2\sqrt{p}, 2\sqrt{p})\cap \mathbb{Z}\right)
\end{equation}
and denote an element of $U_p$ by $(u_{p,1}, u_{p,2})$. We aim to define a probability measure $\mu_p$ on $U_p$ such that almost all sequences $((u_{p,1}, u_{p,2}))$ resemble the trace sequence $((a_{E_1}(p), a_{E_2}(p)))$ in some aspects. Here, the almost all is meant with respect to the product measure $\prod_{p} \mu_p$ on the product space $\prod_{p} U_p$. More precisely, given an integer $m>1$ we will equip $U_p$ with a probability measure $\mu_{(p, {E_1}, E_2, m)}$ such that for almost all sequences $((u_{p,1}, u_{p,2}))$ the behaviour of $((u_{p,1}, u_{p,2}))$ modulo $m$ is consistent with the behaviour of the trace sequence $((a_{E_1}(p), a_{E_2}(p)))$ modulo $m$. Furthermore, for almost all sequences $((u_{p,1}, u_{p,2}))$ we want the distribution of the normalized sequence $((\tfrac{u_{p,1}}{2\sqrt{p}}, \tfrac{u_{p,2}}{2\sqrt{p}})) $ in the square $[-1, 1]\times [-1, 1]$ to be consistent with the distribution of the normalized trace sequence $( (\tfrac{a_{E_1}(p)}{2\sqrt{p}},   \tfrac{a_{E_2}(p)}{2\sqrt{p}})  )$.

We next recall an important property of the trace sequence modulo an integer $m$. For $i=1, 2$, let $E_i[m]$ denote the subgroup consisting of points of $E_i$ of order dividing $m$,  and let $K_m=\QQ(E_1[m], E_2[m] )$ be the $m$-division field associated to $E_1$ and $E_2$ (i.e. the field obtained by joining the coordinates of points in $E_1[m]$ or $E_2[m]$ to $\mathbb{Q}$). For $i=1, 2$, let $N_i$ denote the conductor of $E_i$. Then, by \cite[Theorem 7.1]{Sil}, a prime $p\nmid mN_1 N_2$ is unramified in the Galois extension of $K_m/\QQ$. Moreover, the Frobenius conjugacy class $\sigma_p$ in $K_m/\QQ$ can be identified by a pair $(\sigma_{p, 1}, \sigma_{p, 2})$, where $\sigma_{p, i}$ is the Frobenius conjugacy class at $p$ for the extension $\QQ(E_i[m])/\QQ$. Note that ${\rm det}(\sigma_{p,1})={\rm det}(\sigma_{p,2})=p$, thus $(\sigma_{p, 1}, \sigma_{p, 2})\in \Delta(\ZZ/m\ZZ)$.  It is known that for $p\nmid mN_1N_2$ we have
$$(a_{E_1}(p), a_{E_2}(p))\equiv (t_1, t_2) \imod m \iff ({\rm tr}(\sigma_{p, 1}), {\rm tr}(\sigma_{p, 2}))=(t_1, t_2)~~{\rm in}~~\ZZ/m\ZZ.$$
Thus by an application of the Chebotarev density theorem we have
\begin{equation}
\label{Cheb}
\lim_{x\rightarrow\infty} \frac{\#\{p\leq x;~p\nmid mN_1 N_2~{\rm and}~(a_{E_1}(p), a_{E_2}(p))\equiv (t_1, t_2) \imod m\}}{\pi(x)}\\
=\frac{|{G_{E_1, E_2}}(m)_{t_1, t_2}|}{|G_{E_1,E_2}(m)|},\end{equation}
where $\pi(x)$ denotes the prime counting function.

Another signifying property of the trace sequence is the law governing the distribution of  the values of the normalized trace sequence inside the square $[-1, 1]\times [-1, 1]$. More precisely, for a region $R\subset [-1, 1]\times [-1, 1]$, we expect that 
\begin{equation}
\label{S-T}
\lim_{x\rightarrow \infty}    \frac{ \#\left\{p\leq x;~      \left(\frac{a_{E_1}(p)}{2\sqrt{p}}, \frac{a_{E_2}(p)}{2\sqrt{p}}\right)  \in R          \right\} }{\pi(x)} = \frac{4}{\pi^2}{\iint}_{R} \sqrt{1-u^2} \sqrt{1-v^2} dudv.
\end{equation}
In other words, the expected distribution is the two-dimensional joint Sato-Tate distribution. Ram Murty and Pujahari \cite[Proposition 2.1]{MP} have shown, as a consequence of the recent progress in the proof of the Sato-Tate conjecture, that the joint Sato-Tate distribution holds for two Hecke eigenforms with at least one not of CM type, provided that one is not a Dirichlet twist of the other one. Thus the truth of \eqref{S-T} is known for two non-CM,  non-$\mathbb{Q}^{ab}$-isogenous  elliptic curves, where $\mathbb{Q}^{ab}$ denotes the maximal abelian extension of $\mathbb{Q}$.

From now on for simplicity we set ${\bf E}=(E_1, E_2)$, ${\bf t}=(t_1, t_2)$, and ${\bf u}_p =(u_{p,1}, u_{p,2})$.
Our initial goal is to equip $U_p$ with a probability measure such that for almost all sequences $({\bf u}_p)\in\prod_p{U_p}$ the following two principles hold.
\medskip\par

\noindent {\bf Principle 1 (Consistency with the Chebotarev density theorem)} 
For fixed ${\bf t}\in \mathbb{Z} \times \mathbb{Z}$ and integer $m>1$, we have 
$$
\lim_{x\rightarrow\infty} \frac{\#\{p\leq x;~{\bf u}_p\equiv {\bf t} \imod{m}\}}{\pi(x)}\\
=\frac{|{G_{{\bf E}}}(m)_{{\bf t}}|}{|G_{{\bf E}}(m)|}.
$$

\noindent{\bf Principle 2 (Consistency with the Sato-Tate distribution)}
Let $R\subset [-1,1]\times[-1,1]$. Then we have

{$$ \lim_{x\rightarrow \infty}    \frac{ \#\left\{p\leq x;~    \tfrac{1}{2\sqrt{p}}{\bf u}_p   \in R          \right\} }{\pi(x)} = \frac{4}{\pi^2}{\iint}_{R} \sqrt{1-u^2} \sqrt{1-v^2} dudv. $$}

We next propose our choice of the probability measure on $U_p$.

\subsection{The probability space $(U_p, \mu_{(p,{\bf E}, m)})$:}
\label{p-space}
For ${\bf E}$, ${\bf u}_p$, and  a fixed integer $m>1$, we let 
$$\mu_{(p, {\bf E}, m)}({\bf u}_p):=c_p f_\infty({\bf u}_p, p) f{({\bf u}_p, {\bf E}, m)},$$
where
$$f_\infty({\bf u}_p, p):= \frac{1}{\pi^2 p} \sqrt{1-\left(\frac{u_{p,1}}{2\sqrt{p}}\right)^2} \sqrt{1-\left(\frac{u_{p,2}}{2\sqrt{p}}\right)^2},$$
$$
f{({\bf u}_p, {\bf E}, m)}:= \frac{m^2\cdot |{{G_{\bf E}}(m)}_{{\bf u}_p} | }{{|G_{\bf E}}(m)|},$$
and $c_p$ is a constant such that 
\begin{equation}
\label{prob-relation}
\sum_{{\bf u}_p\in U_p} \mu_{(p, {\bf E}, m)}({\bf u}_p)=1.
\end{equation}

The measure $\mu_{(p, {\bf E}, m)}$ is a well-defined probability measure on $U_p$. The following lemma provides information on  this measure as $p\rightarrow \infty$. Parts (ii) and (iii) of this lemma play important roles in proving that Principles 1 and 2 hold for almost all sequences  $(\textbf{u}_p)$ in  $\prod_{p} U_p$ (see Proposition \ref{consistency}). 

\begin{lem}
\label{model-lemma}
The following assertions hold.

(i)  $\lim_{p\rightarrow \infty} c_p=1$.

(ii) For a fixed ${\bf t}\in \mathbb{Z}\times \mathbb{Z}$ we have
$$\lim_{p\rightarrow \infty} \mu_{(p,{\bf E},m)}\left(\left\{{\bf u}_p\in U_p;~{\bf u}_p \equiv {\bf t} \imod{m}     \right\}  \right)=\frac{|{G_{\bf E}(m)}_{{\bf t}}|}{|G_{{\bf E}}(m)|}.$$

(iii) Let $R\subset [-1,1]\times[-1,1]$. Then we have
$$
\lim_{p\rightarrow \infty} \mu_{(p,{\bf E},m)} \left(\left\{ {\bf u}_p\in U_p;~ \tfrac{1}{2\sqrt{p}}{\bf u}_p  \in R          \right\} \right)={\frac{4}{\pi^2}{\iint}_{R} \sqrt{1-x^2} \sqrt{1-y^2} dxdy.}
$$

\end{lem}
\begin{proof}
(i) We first observe that $$\lim_{p\rightarrow \infty} \frac{1}{4p} \sum_{{\bf u}_p\in U_p} \frac{4}{\pi^2} \sqrt{1-\left(\frac{u_{p,1}}{2\sqrt{p}}\right)^2} \sqrt{1-\left(\frac{u_{p,2}}{2\sqrt{p}}\right)^2}= \int_{-1}^{1} \int_{-1}^{1} \frac{4}{\pi^2}\sqrt{1-x^2} \sqrt{1-y^2} dx dy=1.$$
Thus for a positive integer $m$ we have that 
$$\lim_{p\rightarrow\infty} \sum_{\substack{{\bf u}_p\in U_p\\  {\bf u}_p \equiv {\bf t} \imod{m}}} f_\infty({\bf u}_p, p)=\frac{1}{m^2}.$$
This yields
\begin{equation}
\label{going-back}
\lim_{p\rightarrow\infty} \sum_{\substack{{\bf u}_p\in U_p\\  {\bf u}_p \equiv {\bf t} \imod{m}}} f_\infty({\bf u}_p, p) f{({\bf u}_p, {\bf E}, m)} 
=\frac{|{{G_{\bf E}}(m)}_{{\bf t}} | }{{|G_{\bf E}}(m)|}.
\end{equation}
Summing both sides of the above identity over ${\bf t}=(t_1, t_2)$, where $t_1$ and $t_2$ vary over congruence classes modulo $m$, results in
$$\lim_{p\rightarrow\infty} \frac{1}{c_p}\sum_{{\bf u}_p\in U_p} \mu_{(p, {\bf E}, m)}({\bf u}_p)=\sum_{{\bf t} \imod m}\frac{|{{G_{\bf E}}(m)}_{{\bf t}}  |}{{|G_{\bf E}}(m)|} =1.$$
Now by applying \eqref{prob-relation} in the above identity, we conclude that $\lim_{p\rightarrow \infty} c_p=1$.

(ii) This follows from \eqref{going-back}, part (i), and definition of $\mu_{(p, {\bf E}, m)}$. 

(iii) The proof is similar to part (i). Observe that 
$$\lim_{p\rightarrow \infty} \frac{1}{4p} \sum_{\tfrac{1}{2\sqrt{p}}{{\bf u}_p}\in R} \frac{4}{\pi^2} \sqrt{1-\left(\frac{u_{p,1}}{2\sqrt{p}}\right)^2} \sqrt{1-\left(\frac{u_{p,2}}{2\sqrt{p}}\right)^2}= \iint_{R}\frac{4}{\pi^2}\sqrt{1-x^2} \sqrt{1-y^2} dx dy.$$
Following steps identical to part (i), we have
$$\lim_{p\rightarrow\infty} \frac{1}{c_p}\sum_{\tfrac{1}{2\sqrt{p}}{{\bf u}_p}\in R} \mu_{(p, {\bf E}, m)}({\bf u}_p) = \iint_{R}\frac{4}{\pi^2}\sqrt{1-x^2} \sqrt{1-y^2} dx dy .$$
By part (i) we have  $\lim_{p\rightarrow\infty} c_p =1$ and the result follows.\end{proof}
In order to show that Principles 1 and 2 hold for almost all sequences $({\bf u}_p)\in \prod_{p} U_p$ we need to appeal to the law of large numbers in probability theory.
\subsection{The law of large numbers:}
For $n\in \mathbb{N}$, let $(U_n, \mu_n)$ be a probability space. Let $X_n$ be a real random variable defined on $U_n$. Define $U:=\prod_n U_n$ and $\mu:=\prod_n \mu_n$.
The following theorem which is one of several versions of the strong law of large numbers is due to Kolmogorov.
\begin{thrm}[\bf Strong law of large numbers]
\label{SLLN}
Let $(X_n)$ be a sequence of square integrable independent random variables with expectation sequence $({\rm E}[X_n])$ and variance sequence $({\rm Var}[X_n])$. Let $(b_n)$ be an increasing sequence of positive real numbers such that $\lim_{n\rightarrow \infty}b_n=\infty$. Assume that $$\sum_{n=1}^{\infty} \frac{{\rm Var}[X_n]}{b_n^2} <\infty.$$ 
Then
$$\lim_{N\rightarrow\infty} \frac{1}{b_N} \left( \sum_{n=1}^N X_n -\sum_{n=1}^N {\rm E}[X_n] \right) \rightarrow 0,$$
for almost all sequences $(u_n)\in U$. 
\end{thrm}
\begin{proof}
See \cite[p. 187, Theorem 4.5.1]{Ito}.
\end{proof}
The following is a direct corollary of the strong law of large numbers.
\begin{cor}
\label{cor-SLLN}
If 
$$\lim_{n\rightarrow \infty} E[X_n]=L~~{\rm and}~~\sum_{n=1}^{\infty} \frac{{\rm E}[X_n^2]}{n^2} <\infty,$$ 
then
$$\lim_{N\rightarrow\infty} \frac{1}{N}  \sum_{n=1}^N X_n =L,$$
for almost all sequences $(u_n)\in U$.  
\end{cor}
\begin{proof}
In Theorem \ref{SLLN}, let $b_n=n$ and note that ${\rm Var}[X_n]={\rm E}[X_n^2]-{\rm E}[X_n]^2$. Then the conditions of Theorem \ref{SLLN} are satisfied and the result follows since the sequence of averages of a sequence converging to $L$ also converges to $L$.  
\end{proof}

In our work we consider the following special instances of the above theorem and its corollary. For a prime $p$, let $(U_p, \mu_p)$ be a probability space. In this section $U_p$ is a general set that may be different from \eqref{up}. Let $S_p\subset U_p$ be measurable. Let $X_p: U_p \rightarrow \{0,1\}$ be such that 
\begin{equation}
\label{random variable}
X_p(u_p)=\begin{cases}
1&{\rm if}~~u_p\in S_p,\\
0&{\rm if}~~u_p\not\in S_p.
\end{cases}
\end{equation}

Observe that
$$E[X_p]=E[X_p^2]=\mu_p(S_p).$$
The following version of Corollary \ref{cor-SLLN} is essentially Theorem 2.1 of  \cite{LT}.
\begin{prop}
\label{prop-SLLN}
With the above notation, if 
$$\lim_{p\rightarrow \infty} \mu_p(S_p)=L,$$
then
$$\lim_{N\rightarrow\infty} \frac{1}{\pi(N)} \#\{p\leq N;~u_p\in S_p\}=L,$$
for almost all sequences $(u_p)\in U$. 
\end{prop}

\begin{proof}
Let $p_n$ be the $n$-th prime. Then for any $n\in \mathbb{N}$, let $X_n:=X_{p_n}$,  where $X_{p_n}$ is the random variable defined in \eqref{random variable}. Since $E[X_n]=E[X_n^2]=\mu_{p_n} (S_{p_n})$ and $\lim_{n\rightarrow \infty} \mu_{p_n} (S_{p_n})=L$.
Then by Corollary \ref{cor-SLLN} we have
$$\lim_{N\rightarrow\infty} \frac{1}{N} \#\{p\leq p_N;~u_p\in S_p\}=L.$$\end{proof}

The following proposition is a direct consequence of Theorem \ref{SLLN}.

\begin{prop}
\label{consequence}
With the notation as before, for $C>0$, let 
$$ \mu_p(S_p)\sim \frac{C}{p},\textrm{ as } p\rightarrow \infty.$$
Then, as $N\rightarrow \infty$,
$$\#\{p\leq N;~u_p\in S_p\}\sim C \sum_{p\leq N} \frac{1}{p},$$
for almost all sequences $(u_p)\in \prod_{p}U_p$. 
\end{prop}
\begin{proof}
Let $p_n$ and $X_n$ be defined as in Proposition \ref{prop-SLLN}. For $\epsilon>0$, let $b_n=(\log\log{p_n})^{1/2+\epsilon/2}$. Then $b_n$ is an increasing sequence of real numbers such that $\lim_{n\rightarrow \infty} b_n=\infty$ and 
$$\sum_{n=1}^{\infty} \frac{{\rm Var}[X_n]}{b_n^2} =\sum_{p} \frac{\mu_p(S_p)-\mu_p(S_p)^2}{(\log\log{p})^{1+\epsilon}}\ll \sum_{p} \frac{1}{p(\log\log{p})^{1+\epsilon}}<\infty ,$$ 
by our assumption that $\mu_p(S_p)\ll 1/p$. Since $E[X_n]=\mu_{p_n}(S_{p_n})$, by Theorem \ref{SLLN}, we have
$$\sum_{n=1}^{N} X_n=\#\{p\leq p_N;~u_p\in S_p\}= \sum_{n=1}^{N} \mu_{p_n}(S_{p_n})+o\Big((\log\log{p_N})^{1/2+\epsilon/2}\Big),$$
for almost all sequences $(u_p)\in \prod_p U_p$.  Thus, since $\mu_{p_n}(S_{p_n})\sim C/{p_n}$, we have
$$\#\{p\leq p_N;~u_p\in S_p\}\sim C\sum_{p\leq p_N} \frac{1}{p},$$
as $N\rightarrow\infty$,
for almost all sequences $(u_p)\in \prod_p U_p$. 
\end{proof}

\subsection{The probabilistic model}
We now return to the notation of Section \ref{p-space} and let $U_p$ be defined as in \eqref{up}. As a first task, by employing the results of the previous section, we prove that the model for the distribution of the trace sequence modulo $m$ described in Section \ref{p-space} is consistent with the Chebotarev density theorem and the Sato-Tate distribution (Principles 1 and 2). We have the following. 
\begin{prop}
\label{consistency}
With the notation of Section \ref{p-space}, the following assertions hold.

(i) For fixed ${\bf t}\in \mathbb{Z}\times \mathbb{Z}$ we have
$$\lim_{N\rightarrow \infty}    \frac{ \#\{p\leq N;~   {\bf u}_p \equiv {\bf t}\imod{m} \} }{\pi(N)}  =\frac{|{G_{\bf E}(m)}_{{\bf t}}|}{|G_{{\bf E}}(m)|},$$
for almost all sequences $({\bf u}_p)\in \prod_p U_p$.

(ii) Let $R\subset [-1,1]\times[-1,1]$. Then we have {
$$ \lim_{N\rightarrow \infty}    \frac{ \#\left\{p\leq N;~     \tfrac{1}{2\sqrt{p}}{ {\bf u}_p}  \in R          \right\} }{\pi(N)} = \frac{4}{\pi^2}{\iint}_{R} \sqrt{1-x^2} \sqrt{1-y^2} dxdy, $$}
for almost all sequences $({\bf u}_p)\in \prod_p U_p$.
\end{prop}
\begin{proof}
This is an immediate consequence of Proposition \ref{prop-SLLN} and parts (ii) and (iii) of Lemma \ref{model-lemma}.
\end{proof}

Next observe that for two integers $a, b$ we have
$$a=b \iff a\equiv b \imod{m},~~ {\rm as}~~  
m\widetilde{\rightarrow} \infty.$$ 
Thus the congruences modulo $m$ as $m\widetilde{\rightarrow} \infty$ capture the values of the trace sequence. In conclusion,  as a model for the distribution of primes with given trace for two elliptic curves we consider the product space $U=\prod_p U_p$  in which each $U_p$ is equipped with a probability measure $\mu_{(p, {\bf E})}$ defined by
$$\mu_{(p, {\bf E})} ({\bf u}_p)=\lim_{m \widetilde{\rightarrow} \infty} \mu_{(p, {\bf E}, m)}({\bf u}_p).$$

We now state an application of Proposition \ref{consequence}, which provides probabilistic evidence of  Conjecture \ref{conj: langtrot 2}.
\begin{prop}
With $(U_p, \mu_{(p, {\bf E})})$  as above, for fixed ${\bf t}$, as $N\rightarrow \infty$, we have
$$\#\{p\leq N;~{\bf u}_p= {\bf t} \}\sim \frac{1}{\pi^2 } \lim_{m \widetilde{\rightarrow} \infty} \frac{m^2\cdot |{{G_{\bf E}}(m)}_{\bf t}  |}{|{G_{\bf E}}(m)|} \log\log{N},$$
for almost all sequences $({\bf u}_p)\in \prod_p U_p$.
\end{prop}

\begin{proof}
From the definition of $\mu_{(p, {\bf E}, m)}$ in Section \ref{p-space}, we deduce that
$$\mu_{(p, {\bf E})} ({\bf u}_p)=\lim_{m \widetilde{\rightarrow} \infty} \mu_{(p, {\bf E}, m)}({\bf u}_p)
\sim {\frac{1}{p\pi^2  }\lim_{m \widetilde{\rightarrow} \infty} \frac{m^2 \cdot |{{G_{\bf E}}(m)}_{{\bf u}_p}|  }{|{G_{\bf E}}(m)|}},$$
as $p\rightarrow \infty$.
Thus, considering $\mu_{(p, {\bf E})}$ as $\mu_p$ and $S_p=\{{\bf t}\}$ in Proposition \ref{consequence},  for almost all sequences $({\bf u}_p)\in U$ we have that
$$\#\{p\leq N;~{\bf u}_p= {\bf t} \}\sim \frac{1}{\pi^2 } \lim_{m \widetilde{\rightarrow} \infty} \frac{m^2 \cdot |{{G_{\bf E}}(m)}_{\bf t}  |}{|{G_{\bf E}}(m)|} \sum_{p\leq N} \frac{1}{p},$$
as $N\rightarrow \infty$.
\end{proof}

\section{Proof of Theorem \ref{thrm: ctonetwo}}\label{section: theorem onethree proof}
\begin{proof}[Proof of Theorem \ref{thrm: ctonetwo}]
We first show that in $\eqref{defn: ctonettwo delta}$ we have
\begin{equation}
\label{identity}
\frac{\ell^{2k} \cdot |{\Delta(\ZZ/\ell^k\ZZ)}_{t_1, t_2}|}{| \Delta(\ZZ/\ell^k\ZZ)|}=\frac{\Sts}{(\ell-1)^3(\ell+1)^2\ell^{5k-5}}.
\end{equation}
Note that from \cite[Proposition A.6]{CDKS} we have that
$\#\{g\in \text{GL}_2(\ZZ/\ell^k\ZZ);~ \det(g)=u\}= \ell^{3k-2}(\ell^2-1)$ and thus
\begin{multline}
| \Delta(\ZZ/\ell^k\ZZ)|=\sum_{u\in\zstar}\# \{(g_1,g_2)\in \text{GL}_2(\ZZ/\ell^k\ZZ)\times\text{GL}_2(\ZZ/\ell^k\ZZ);~ \det(g_1)=\det(g_2)=u\}\\
=\sum_{u\in\zstar} \#\{g\in \text{GL}_2(\ZZ/\ell^k\ZZ);~ \det(g)=u\}^2=\varphi(\ell^k)(\ell^{3k-2}(\ell^2-1))^2.\label{deltagl}
\end{multline}
Similarly, 
\begin{align}
|{\Delta(\ZZ/\ell^k\ZZ)}_{t_1, t_2}|=&\sum_{u\in\zstar}\# \{(g_1,g_2)\in \Delta(\ZZ/\ell^k\ZZ);~  \tr(g_1)=t_1\text{ and} \tr(g_2)=t_2 \}\nonumber\\
=&\sum_{u\in\zstar} \mtone\mttwo = \Sts.\label{deltats}
\end{align} 
Thus the identity \eqref{identity} holds by combining $\eqref{deltagl}$ and $\eqref{deltats}$. 

Next let  $\Sk={\Sts}/{\ell^{5k-5}}$. Following the approach of \cite[Theorem 1.6]{DKS}, we show that $\lim_{k\rightarrow \infty} \Sk$ exists for all primes $\ell$ by establishing that $(\Sk)_{k\geq 1}$ is a Cauchy sequence. By abuse of notation let $u$ denote an integer representative of the class $u\in \ZZ/\ell^k \ZZ$. Then, for $i=1, 2$, set $D_i:=t_i^2-4u$. By  
\cite[Theorem 3.2]{DKS}, 
we have 
\begin{equation}
\label{m-definition}
m(t_i, u; \ell^k)=\ell^{2k}+\ell^{2k}\sum_{j=1}^{\min\{k,\nu_\ell(D_i)+1\}}\frac{N_{D_i}(\ell^j)-N_{D_i}(\ell^{j-1})}{\ell^j}, 
\end{equation}
where
$$N_{D_i}(m):=\frac{\#\{x\imod{4m} : x^2\equiv {D_i}\imod{4m}\}}{2}.$$ 
Now for positive integers $r$, $s$ with
 $r>s$, by employing \eqref{m-definition}, we have that 
	\begin{align}
	\mathcal{S}_r&=\frac{\ell^{4r}}{\ell^{5r-5}}\sum_{u \in (\ZZ/\ell^r\ZZ)^*}\prod_{i=1}^2\left(1+\sum_{j_i=1}^{\min\{r,\nu_\ell(D_i)+1\}}\frac{N_{D_i}(\ell^{j_i})-N_{D_i}(\ell^{j_i-1})}{\ell^{j_i}}\right)\nonumber\\
	&=\ell^{5-r}\sum_{u \in (\ZZ/\ell^r\ZZ)^*}\prod_{i=1}^2\left(1+\sum_{j_i=1}^{\min\{s,\nu_\ell(D_i)+1\}}\frac{N_{D_i}(\ell^{j_i})-N_{D_i}(\ell^{j_i-1})}{\ell^{j_i}}+\sum_{j_i=s+1}^{\min\{r,\nu_\ell(D_i)+1\}}\frac{N_{D_i}(\ell^{j_i})-N_{D_i}(\ell^{j_i-1})}{\ell^{j_i}}\right),\label{sknd}
	\end{align}
	where the second sum (the sum starting from $s+1$) is empty if $\nu_\ell(D_j)+1\leq s$.
	We have that the main term of $\eqref{sknd}$ is
	\begin{multline*}
	\ell^{5-r}\sum_{u \in (\ZZ/\ell^r\ZZ)^*}\prod_{i=1}^2\left(1+\sum_{j_i=1}^{\min\{s,\nu_\ell(D_i)+1\}}\frac{N_{D_i}(\ell^{j_i})-N_{D_i}(\ell^{j_i-1})}{\ell^{j_i}}\right)\\
	=\ell^{5-s}\sum_{u \in (\ZZ/\ell^s\ZZ)^*}\prod_{i=1}^2\left(1+\sum_{j_i=1}^{\min\{s,\nu_\ell(D_i)+1\}}\frac{N_{D_i}(\ell^{j_i})-N_{D_i}(\ell^{j_i-1})}{\ell^{j_i}}\right)=\mathcal{S}_s,
	\end{multline*}
	since the inner sum only depends on the condition $D_i \imod{4\ell^s}$. For the second sum in $\eqref{sknd}$, either $\nu_\ell(D_i)+1\leq s$ and the sum is empty or $\min\{\nu_\ell(D_i)+1,r\}\geq s+1$, in which case we use the bound $N_{D_i}(\ell^{j_i})\ll \ell^{j_i/2}.$
	Hence $$\mathcal{S}_r=\mathcal{S}_s+O\left(\ell^{5-r}\sum_{u \in (\ZZ/\ell^r\ZZ)^*}\frac{1}{\ell^{\frac{s+1}2}}\right)=\mathcal{S}_s+o(1),$$ 
as $s\rightarrow \infty$.
We conclude that $(\Sk)_{k\geq 1}$ is a Cauchy sequence and thus $\lim_{k\rightarrow\infty} S(t_1, t_2;\ell^k)/\ell^{5k-5}$ exists.

Finally, in Section \ref{section: DKS}, we will show that 
$$\Delta_\ell: = -1+\frac{1}{(\ell-1)^3(\ell+1)^2}\lim_{k\rightarrow \infty}\frac{\Sts}{\ell^{5k-5}}$$
satisfies the general framework of Theorem 4.2 of \cite{DKS} and thus, by \cite[Formula (6.10)]{DKS},
we have $\Delta_\ell \ll 1/\ell^{3/2}$ for sufficiently large $\ell$. Therefore $c_{t_1, t_2}$ is given by the convergent Euler product in Theorem \ref{thrm: ctonetwo}.
\end{proof}

\section{Proofs of Theorem \ref{samet} and Proposition \ref{proved-cases}}\label{section: ltcomps}
For the duration of this section  we use the following notation.  For integers $t$ and $u$ set $D(t,u):=t^2-4u$. Let $\ell$ be an odd prime or let $\ell=2$ and $t$ be odd. Then for a positive integer $k$, we set 
$$\nu_{\ell,k}(D(t,u)):=\begin{cases}\nu_{\ell}(D(t, u)) &\text{ if }\ell^k\nmid D(t, u),\\ k &\text{ if }\ell^k \mid D(t, u).\end{cases}$$
Also, for $\ell=2$ and even $t$, we set
 $$ \nu_{2,k}(D(t,u)):=\begin{cases}\nu_{2}(D(t, u)) &\text{ if }2^{k+2}\nmid D(t, u),\\ k+2 &\text{ if }2^{k+2} \mid D(t, u).\end{cases}$$ 
 We begin with an explicit representation of the quantity $\mt$ defined in $\eqref{def: mtplk}$.
\begin{thrm}[\textbf{Gekeler and Castryck-Hubrechts}]\label{thrm: ch-thrm-five}
If $\ell>2$, letting $n:= \nu_{\ell, k}(D(t, u))$, we have that
	\begin{equation*}
	\mt=\begin{cases}
	\ell^{2k}+\ell^{2k-1} & \eif  n \emph{\text{ even}}, n<k, \eand \left(\frac{D(t,u)/\ell^n}{\ell}\right)=1, \\
	\ell^{2k}+\ell^{2k-1}-2\ell^{2k-\frac{n}2-1} & \eif n \emph{\text{ even}}, n<k, \eand \left(\frac{D(t,u)/\ell^n}{\ell}\right)=-1, \\
	\ell^{2k}+\ell^{2k-1}-(\ell+1)\ell^{2k-\frac{n +3}2} & \eif n \emph{\text{ odd}} \eand n<k,\\
	\ell^{2k}+\ell^{2k-1}-\ell^{\frac{3k}2+\frac{1-(-1)^k}4-1} & \eif n=k.
	\end{cases}
	\end{equation*}
	If $\ell=2$ and $t$ is odd, we have that $$m(t,u;2^k)=	2^{2k-1}.$$ 	
	If $\ell=2$ and $t$ is even, we set $n:=\nu_{2,k}(D(t, u))$ and we have that
	\begin{equation*}
	m(t,u;2^k)=\begin{cases}
	2^{2k}+2^{2k-1}-2^{\frac{3k}2+\frac{1-(-1)^k}4-1}& \eif n=k+2,\\
	2^{2k}+2^{2k-1}-3\cdot 2^{2k-\frac{n+1}{2}} & \eif   n \eodd \eand n< k+2.
	\end{cases}
	\end{equation*}
	Furthermore, if $0<n<k+2$ is even, we write $r:=D(t,u)/2^n $. In this case we have that 
	\begin{equation*}
	m(t,u;2^k)=\begin{cases}
	2^{2k}+2^{2k-1}-2^\frac{3k-1}2 & \eif n=k+1,\\
	2^{2k}+2^{2k-1}-2^{\frac{3k}2-1}& \eif n=k \eand r\equiv 1\imod{4},\\
	2^{2k}+2^{2k-1}-3\cdot 2^{\frac{3k}2-1}& \eif n=k \eand r\equiv 3\imod{4},\\
	2^{2k}+2^{2k-1}-3\cdot 2^{2k-\frac{n}{2}-1} & \eif n<k \eand r\equiv 3\imod{4},\\
	2^{2k}+2^{2k-1} & \eif n<k \eand r\equiv 1\imod{8},\\
		2^{2k}+2^{2k-1}-2^{2k-\frac{n}{2}} & \eif n<k \eand r\equiv 5\imod{8}.

	\end{cases}
	\end{equation*}
\end{thrm}

\begin{proof}
See \cite[Theorem 4.4]{G} for  a proof when $k\geq 2[n/2]+2$. For the general case see \cite[Theorem 5]{CH}, where there is no restriction on $k$. The case considered in \cite{CH} is written for $u$ that is a prime power relatively prime to $\ell$, but the proof is independent of these conditions on $u$. 
\end{proof}

Our next goal is to calculate $\Stt$ by applying Theorem \ref{thrm: ch-thrm-five}. To do this we count the number of $u\in \zstar$ subject to the various conditions stated above. This results in a number of different cases, so we consider the primes $\ell>2$ and $\ell=2$ separately.
\subsection{The case $\ell>2$}
We break this case into two subcases, $\ell \nmid t$ and $\ell \mid t$. In order to calculate $\Stt$ when $(\ell,2t)=1$, we first count the number of $u\in \zstar$ such that $n:= \nu_{\ell, k}(D(t,u))$ is a fixed integer. 

\begin{lem}\label{kroncount}
Let $\ell$ be a prime and let $t$ be an integer such that $(\ell,2t)=1$.  Then for $k\geq 1$ we have that $$\#\Bigg\{u\in \zstar;~ \nu_{\ell, k}(D(t,u))=n \Bigg\}=\begin{cases}
\frac{\varphi(\ell^{k-2i})}2-\ell^{k-1}\chi_{\{ 0 \}} (n)& \eif  n=2i, n<k, \eand \left(\frac{D(t,u)/\ell^n}{\ell}\right)=1,\\
\frac{\varphi(\ell^{k-2i})}2 & \eif  n=2i, n<k, \eand \left(\frac{D(t,u)/\ell^n}{\ell}\right)=-1,\\
	\varphi(\ell^{k-2i-1}) & \eif n=2i+1 \eand n<k,\\
	1 & \eif n=k,
	\end{cases}$$
	where $\chi_{\{ 0 \}} $ is the indicator function taking the value one if $n=0$ and zero otherwise.
\end{lem}

\begin{proof}
We begin with the fourth case. For $n= k$, since $D(t,u)=t^2-4u$ and $\ell$ is odd, we have that
\begin{equation}
\label{kcase}
\#\Bigg\{u\in \zstar;~  D(t,u)\equiv 0\imod{\ell^k}\Bigg\}=1.
\end{equation}

For the third case we have that $0< n=2i+1 <k$ and thus by employing \eqref{kcase} we have
	\begin{align}
	\#\{u\in \zstar;~ n= 2i+1\}=&~\#\{u\in \zstar;~ D(t,u)\equiv 0 \imod{\ell^{2i+1}}~{\rm and}~D(t,u)\not\equiv 0 \imod{\ell^{2i+2}}\}\nonumber\\
	=&~\ell^{k-2i-1}\cdot\#\{u\in (\ZZ/\ell^{2i+1}\ZZ)^*;~ D(t,u)\equiv 0 \imod{\ell^{2i}}\}\nonumber\\
	&-\ell^{k-2i-2}\cdot\#\{u\in (\ZZ/\ell^{2i+2}\ZZ)^*;~ D(t,u)\equiv 0  \imod{\ell^{2i+2}}\}\nonumber\\
	=&~\ell^{k-2i-1}-\ell^{k-2i-2}=\varphi(\ell^{k-2i-1}).\label{casethree}
	\end{align}

In the second case we have that $0\leq n=2i<k$ and
	\begin{align*}
	&\#\Bigg\{u\in (\ZZ/\ell^k\ZZ)^*;~n=2i<k, \left(\frac{D(t,u)/\ell^{2i}}{\ell}\right)=-1 \Bigg\}\\
		=&~\ell^{k-2i-1}\cdot\#\Bigg\{u\in (\ZZ/\ell^{2i+1}\ZZ)^*;~\left(\frac{D(t,u)/\ell^{2i}}{\ell}\right)=-1, D(t,u)\equiv 0 \imod{\ell^{2i}},~{\rm and}~ D(t,u)\not \equiv 0 \imod{\ell^{2i+1}}\Bigg\}\\
	=&~\ell^{k-2i-1}\cdot\frac{\varphi(\ell)}{2}=\frac{\varphi(\ell^{k-2i})}{2}.
	\end{align*}

Finally for the first case, the argument is similar to the second case. One thing to observe is that for $n=2i=0$,
$$\#\Bigg\{u\in (\ZZ/\ell^{2i+1}\ZZ)^*;~\left(\frac{D(t,u)/\ell^{2i}}{\ell}\right)=1, D(t,u)\equiv 0 \imod{\ell^{2i}},~{\rm and}~ D(t,u)\not \equiv 0 \imod{\ell^{2i+1}}\Bigg\}=\frac{\varphi(\ell)}{2}-1,$$
since $(\tfrac{t^2-4u}{\ell})=1$ if $u=0$ in $\ZZ/\ell^k\ZZ$ which is not allowed.
\end{proof}

By applying Lemma \ref{kroncount} with Theorem \ref{thrm: ch-thrm-five} we obtain a formula for $\Sk$ in the case $(\ell,2t)=1$. We note that $\Sk$ does not stabilize for any value of $k$ in this case.

\begin{lem}\label{lemma: sltcoprime}
Let $t$ be an integer and $\ell$ be a prime such that $(\ell,2t)=1$. Then we have, for $k\geq 1$, that 	
$$\frac{\Stt}{\ell^{5k-5}}=\frac{\ell^2(\ell^4-2\ell^2-3\ell-1)}{\ell+1}-\frac{\ell^4}{\ell^{2k}(\ell+1)}.$$
\end{lem}

\begin{proof}
From Theorem \ref{thrm: ch-thrm-five} and Lemma \ref{kroncount} we have that 
\begin{align}
\Stt=&\left(\frac{\varphi(\ell^k)}{2}-\ell^{k-1}+\sum_{\substack{1\leq i< \frac k2}}\frac{\varphi(\ell^{k-2i})}{2}\right)\left(\ell^{2k}+\ell^{2k-1}\right)^2+\sum_{0\leq i< \frac k2}\frac{\varphi(\ell^{k-2i})}{2}\left(\ell^{2k}+\ell^{2k-1}-2\ell^{2k-i-1}\right)^2\nonumber\\
&+\sum_{0\leq i <\frac{k-1}2}\varphi(\ell^{k-2i-1})\left(\ell^{2k}+\ell^{2k-1}-(\ell+1)\ell^{2k-i-2}\right)^2+\left(\ell^{2k}+\ell^{2k-1}-\ell^{\frac{3k}2+\frac{1-(-1)^k}4-1}\right)^2.\label{eqn: scoprimebig}
\end{align}
We can perform the tedious calculation directly or employ a symbolic calculation software such as MAPLE  to simplify $\eqref{eqn: scoprimebig}$ to obtain the result.
\end{proof}
We now consider the case when $\ell\mid t$  and $\ell>2$. More generally, in this case,  we obtain  formulas for $\mathcal{S}_k$  for $t_1$ and $t_2$, where $\ell \mid t_1t_2$.

\begin{lem}\label{lemma: ldivides t}
	Let $\ell>2$ be a prime, let $t_1,t_2$ be integers such that $\ell \mid  (t_1,t_2)$. Then for $k\geq 1$ we have that 
	$$\frac{\Sts}{\ell^{5k-5}}=\ell^2(\ell^2+1)(\ell-1).$$
\end{lem}

\begin{proof}
Since $\ell \mid (t_1,t_2)$, then 
for all $u\in \zstar$ we have that $\nu_{\ell,k}(D(t_i,u))=0$ and hence $\left(\frac{D(t_i,u)}{\ell}\right)\neq 0$ for $i=1,2$. It remains to count the number of $u$'s for which  $\left(\frac{-4u}{\ell}\right)=1$ (-1, respectively).  From Theorem \ref{thrm: ch-thrm-five} we have that
	\begin{align*}
	\Sts&=\sum_{\substack{u\in \zstar \\ \left(\frac{-4u}{\ell}\right)=1 }}\mtone\mttwo+\sum_{\substack{u\in \zstar \\ \left(\frac{-4u}{\ell}\right)=-1 }}\mtone\mttwo\\
	&= \frac{\varphi(\ell^k)}{2}\left(\ell^{2k}+\ell^{2k-1}\right)^2+\frac{\varphi(\ell^k)}{2}\left(\ell^{2k}-\ell^{2k-1}\right)^2=\ell^{5k-3}(\ell^2+1)(\ell-1).
	\end{align*} 
\end{proof}

\begin{lem}\label{lemma: ldivides t_1}
	Let $\ell>2$ be a prime, let $t_1,t_2$ be integers such that $\ell \mid  t_1$ and $\ell \nmid t_2$. Then for $k\geq 1$ we have that 
	$$\frac{\Sts}{\ell^{5k-5}}=\ell^2(\ell^2-1)(\ell-1).$$
\end{lem}

\begin{proof}

For $r=1, 2$, $1\leq j \leq 4$, $k\geq 1$, and $0\leq n \leq k$, let $g_r(j, n)$ be the formula for $m(t_r, u;\ell^k)$ in Theorem \ref{thrm: ch-thrm-five} corresponding to the $j$-th condition for a fixed value of $n$. For example  $g_r(2, n)= \ell^{2k}+\ell^{2k-1}-2\ell^{2k-\frac{n}{2}-1}$, which corresponds to the second condition (i.e. $n$ even, $n<k$, and $\left( \frac{D(t_r, u)/\ell^n}{\ell} \right)=-1$). Using this notation we have
\begin{equation}
\label{S-formula}
S(t_1, t_2; \ell^k)=\sum_{j_1, j_2, n_1, n_2} f(j_1, j_2, n_1, n_2) g_1(j_1, n_1) g_2(j_2, n_2),
\end{equation}
where 
$$ f(j_1, j_2, n_1, n_2)= \#\{u\in (\ZZ/\ell^k\ZZ)^*;~\textrm{Condition }(j_r)~ \textrm{with value}~ n_r~ \textrm{holds for } D(t_r, u),~{\rm for}~r=1, 2\}.$$
Observe that since $\ell\mid t_1$ and $u\in (\ZZ/\ell^k\ZZ)^*$, we have $n_1=\nu_{\ell ,k}(D(t_1, u))=0$. Thus $f(j_1, j_2, n_1, n_2)=0$ for $j_1=3, 4$ or $n_1\neq 0$.

Now suppose that $\ell\equiv 1\imod{4}$.  Then $f(2, 3, 0, n_2)=f(2, 4, 0, n_2)=0$. This is true since in these cases $t_2^2-4u\equiv 0 \imod{\ell^n_2}$ and thus $4u$ is a quadratic residue mod $\ell$. Since $\left(\tfrac{-1}{\ell}\right)=1$ then $\left(\tfrac{t_1^2-4u}{\ell}\right)=\left(\tfrac{-4u}{\ell}\right)=1\neq -1$. A similar argument shows that in this case $f(2, 1, 0, n_2)=f(2, 2, 0, n_2)=0$ for $n_2\neq 0$ as well. 

Following arguments similar to Lemma \ref{kroncount} we can show that, for $i, k>0$, 
\begin{multline*}
f(1, 1, 0, 0)=\frac{\varphi(\ell^k)}{4}-\ell^{k-1},~f(1, 1, 0, 2i)= \frac{\varphi(\ell^{k-2i})}{2},~f(1, 2, 0, 0)=\frac{\varphi(\ell^k)}{4},~f(1, 2, 0, 2i)= \frac{\varphi(\ell^{k-2i})}{2},\\
f(1, 3, 0, 2i+1)={\varphi(\ell^{k-2i-1})},~f(1, 4, 0, k)= 1,~f(2, 1, 0, 0)=\frac{\varphi(\ell^k)}{4},~{\rm and}~f(2, 2, 0, 0)= \frac{\varphi(\ell^{k})}{4}. \end{multline*}
Applying the above in \eqref{S-formula} yields
\begin{align}
S(t_1, t_2; \ell^k)=&\left(\frac{\varphi(\ell^k)}{4}-\ell^{k-1}+\sum_{\substack{1\leq i< \frac k2}}\frac{\varphi(\ell^{k-2i})}{2}\right)\left(\ell^{2k}+\ell^{2k-1}\right)^2+
\frac{\varphi(\ell^k)}{4}\left(\ell^{2k}+\ell^{2k-1}\right)\left(\ell^{2k}-\ell^{2k-1}\right)\nonumber\\
&+\sum_{0\leq i< \frac k2}\frac{\varphi(\ell^{k-2i})}{2}   
\left(\ell^{2k}+\ell^{2k-1}\right)
\left(\ell^{2k}+\ell^{2k-1}-2\ell^{2k-i-1}\right)\nonumber\\
&+\sum_{0\leq i <\frac{k-1}2}\varphi(\ell^{k-2i-1})\left(\ell^{2k}+\ell^{2k-1}\right)\left(\ell^{2k}+\ell^{2k-1}-(\ell+1)\ell^{2k-i-2}\right)\nonumber\\
&+
\left(\ell^{2k}+\ell^{2k-1}\right)\left(\ell^{2k}+\ell^{2k-1}-\ell^{\frac{3k}2+\frac{1-(-1)^k}4-1}\right)+
\frac{\varphi(\ell^k)}{4}\left(\ell^{2k}-\ell^{2k-1}\right)\left(\ell^{2k}+\ell^{2k-1}\right)\nonumber\\
&+\frac{\varphi(\ell^k)}{4}\left(\ell^{2k}-\ell^{2k-1}\right)^2.\nonumber
\end{align}
Simplifying this expression verifies the result.  The proof  for  $\ell\equiv 3\imod{4}$ is similar, considering the fact that in this case the non-zero values of $f(j_1, j_2, n_1, n_2)$ are
\begin{multline*}
f(1, 1, 0, 0)=\frac{\varphi(\ell^k)}{4}-\frac{\ell^{k-1}}{2},~f(1, 2, 0, 0)= \frac{\varphi(\ell^k)}{4}+\frac{\ell^{k-1}}{2},~f(2, 1, 0, 0)=\frac{\varphi(\ell^k)}{4}-\frac{\ell^{k-1}}{2},~f(2, 1, 0, 2i)= \frac{\varphi(\ell^{k-2i})}{2},\\
f(2, 2, 0, 0)=  \frac{\varphi(\ell^k)}{4}-\frac{\ell^{k-1}}{2}    ,~f(2, 2, 0, 2i)= \frac{\varphi(\ell^{k-2i})}{2},~f(2, 3, 0, 2i+1)={\varphi(\ell^{k-2i-1})},~{\rm and}~f(2, 4, 0, k)= 1. \end{multline*}
\end{proof}

\subsection{The case $\ell=2$}

The situation is more complicated when $\ell=2$ and there are several different cases to consider. We first consider cases when $t$ is even and $4\nmid t$. 
\begin{lem}\label{lemma: teven}
Let $t$ be an even integer such that $4\nmid t$. Then for $k\geq 1$  we have that 
 \begin{equation}
 \label{cases}
 \#\{u\in \zstartwo;~ {\nu}_{2, k}(D(t,u))=n \}=\begin{cases}
 1 & \eif n=k+2,\\
\varphi(2^{k-2i+1}) &\eif n=2i+1, n<k+2,\\
1 & \eif n=2i=k+1,\\
1 &\eif n=2i=k \eand \tfrac{D(t,u)}{2^{2i}}\equiv 1\imod{4},\\
1 &\eif n=2i=k \eand \tfrac{D(t,u)}{2^{2i}}\equiv 3 \imod{4},\\
\varphi(2^{k-2i+1}) & \eif 0<n=2i<k \eand \tfrac{D(t,u)}{2^{2i}}\equiv 3 \imod{4}, \\
\varphi(2^{k-2i}) & \eif 0<n=2i<k \eand \tfrac{D(t,u)}{2^{2i}}\equiv 1 \imod{8},\\
\varphi(2^{k-2i}) & \eif 0<n=2i<k \eand \tfrac{D(t,u)}{2^{2i}}\equiv 5 \imod{8}.\\
\end{cases}
\end{equation}
\end{lem}

\begin{proof}
We set $t:=2m$ with $m$ odd and thus, $D(t,u)=4(m^2-u)$. We now consider cases.

In the first case, if $\nu_{2, k}(D(t, u))=k+2$, then $2^k \mid m^2-u$. Therefore, 
$$\#\{u\in (\ZZ/2^k\ZZ)^*;~ \nu_{2, k}(D(t, u))=k+2\}=\#\{u\in (\ZZ/2^k\ZZ)^*;~u\equiv m^2\imod{2^k}\}=1.$$

Next note that if $\nu_{2, k}(D(t, u))=n<k+2$, we have 
\begin{multline}
\label{general-n}
\#\{u\in (\ZZ/2^k\ZZ)^*;~ \nu_{2, k}(D(t, u))=n\}= \\2^{k-(n-2)} \#\{u\in (\ZZ/2^{n-2}\ZZ)^*;~u\equiv m^2 \imod{2^{n-2}}\}- 2^{k-(n-3)} \#\{u\in (\ZZ/2^{n-3}\ZZ)^*;~u\equiv m^2 \imod{2^{n-3}}\}\\=2^{k-(n-2)}-2^{k-(n-3)}=\varphi(2^{k-n+2}).
\end{multline}
Then the second case follows from $\eqref{general-n}$ by setting $n=2i+1$ and the third case follows by setting $n=k+1$.

If $n=k$ and $\frac{D(t, u)}{2^k}\equiv 1 \imod{4}$, we have
$$\#\{u\in (\ZZ/2^k\ZZ)^*;~ \nu_{2, k}(D(t, u))=n\}= \#\{u\in (\ZZ/2^k\ZZ)^*;~ u\equiv m^2 \imod{2^{k-2}}~{\rm and}~ \frac{m^2-u}{2^{k-2}}\equiv 1 \imod{4} \}=1.$$
The case $n=k$ and $\frac{D(t, u)}{2^k}\equiv 3\imod{4}$ can be treated in exactly the same way.

If $0<n<k$ and $\frac{D(t, u)}{2^n}\equiv 3 \imod{4}$, we have
$$\#\{u\in (\ZZ/2^n\ZZ)^*;~ \nu_{2, k}(D(t, u))=n\}= 2^{k-n}\#\{u\in (\ZZ/2^n\ZZ)^*;~ u\equiv m^2 \imod{2^{n-2}}~{\rm and}~ \frac{m^2-u}{2^{n-2}}\equiv 3 \imod{4} \},$$which is $\varphi(2^{k-n+1}).$ The sixth case follows by letting $n=2i$.

If $0<n<k$ and $\frac{D(t, u)}{2^n}\equiv 1 \imod{8}$, we have
$$\#\{u\in (\ZZ/2^n\ZZ)^*;~ \nu_{2, k}(D(t, u))=n\}= 2^{k-n-1}\#\{u\in (\ZZ/2^{n+1}\ZZ)^*;~ u\equiv m^2 \imod{2^{n-2}}~{\rm and}~ \frac{m^2-u}{2^{n-2}}\equiv 1 \imod{8} \},$$which is $\varphi(2^{k-n}).$ The seventh case follows by letting $n=2i$. The eighth case can be dealt with similarly to the seventh case.
\end{proof}
\begin{remark}
{\em Note that for $k=1$ the only non-empty condition in \eqref{cases} is the first condition. For $k=2$ the non-empty conditions in \eqref{cases} are the first and the second one. Also the last three conditions in \eqref{cases} are non-empty only if $k\geq 5$.}
\end{remark}
We now obtain a formula for $\Sk$ in the case that $\ell=2$ and $t$ is even but not a multiple of $4$.  We have that $\Sk$ is not stable in this case as well.
\begin{lem}\label{lemma: stwobad}
	Let $t$ be an even integer such that $4\nmid t$. Then for $k\geq 3$ we have that 	
	$$\frac{\Stwo}{2^{5k-5}}=\frac{103}{6}-\frac{32}{3\cdot 2^{2k}}.$$
\end{lem}

\begin{proof}
We apply Theorem \ref{thrm: ch-thrm-five} together with Lemma \ref{lemma: teven} and note that for a fixed value of $k$, the third case in Lemma \ref{lemma: teven} only occurs when $k$ is odd and $k\geq3$ and the fourth and fifth cases only occur when $k$ is even and $k\geq 4$. Then we have that
	\begin{align}
		\Stwo=&~(2^{2k}+2^{2k-1}-2^{\frac{3k}2+\frac{1-(-1)^k}4-1})^2+
		\sum_{1\leq i< \frac {k+1}2} \varphi(2^{k-2i+1})\left(2^{2k}+2^{2k-1}-3\cdot 2^{2k-i-1}\right)^2\nonumber\\
		&+\tfrac{(1-(-1)^k)}2\left(2^{2k}+2^{2k-1}- 2^{\frac{3k-1}2}\right)^2+\tfrac{(1+(-1)^k)}2\bigg(\left(2^{2k}+2^{2k-1}- 2^{\frac{3k}2-1}\right)^2+\left(2^{2k}+2^{2k-1}-3\cdot 2^{\frac{3k}2-1}\right)^2\bigg)\nonumber \\
		&+\sum_{2\leq i <\frac{k}2}\varphi(2^{k-2i+1})\left(2^{2k}+2^{2k-1}-3\cdot 2^{2k-i-1}\right)^2\nonumber\\
		&+\sum_{2\leq i <\frac{k}2}\varphi(2^{k-2i})\bigg(\left(2^{2k}+2^{2k-1}\right)^2+\left(2^{2k}+2^{2k-1}-2^{2k-i}\right)^2 \bigg).\label{eqn: stwobig}
	\end{align}
By a straightforward calculation, we simplify $\eqref{eqn: stwobig}$ to obtain the result.
\end{proof}

We now consider the cases when $4\mid t$ or $t$ is odd. 
As in Lemma \ref{lemma: teven} we must first count the number of $u\in \zstartwo$ such that $n:= {\nu}_{2,k}(D(t,u))$ is a fixed integer.

\begin{lem}\label{lemma: toddorfour}
Let $t$ be an integer such that $4\mid t$. Then for $k\geq 3$ we have that 
	$$\#\{u\in \zstartwo;~ {\nu}_{2,k}(D(t,u))=n \}=\begin{cases}
	{\varphi(2^{k-1})} &\eif  n=2 \eand \tfrac{D(t,u)}{4}\equiv 3 \imod {4} ,\\
	{\varphi(2^{k-2})} &\eif  n=2 \eand \tfrac{D(t,u)}{4}\equiv 1 \imod {8},\\
	{\varphi(2^{k-2})} &\eif  n=2 \eand \tfrac{D(t,u)}{4}\equiv 5 \imod {8},\\
	0 & \emph{\text{all other cases in }} \eqref{cases}.
	\end{cases}$$
\end{lem}

\begin{proof}
If  $4\mid t$ then  $D(t,u)\equiv \pm 4 \imod{16}$ and thus $n={\nu}_{2,k}(D(t,u))=2$. For $n=2$ and $k\geq 3$, the first five cases in \eqref{cases} do not occur. For the remaining cases, we note that removing the restriction that $4\nmid t$ in the proof of Lemma \ref{lemma: teven} does not change the result when $n=2$.
\end{proof}

\begin{lem}\label{lemma: sltk results}
Let $t$, $t_1$, and $t_2$ be integers. Then for $k\geq 3$ we have that 
	\begin{equation*}
	\frac{\Sttwo}{2^{5k-5}}=\begin{cases}
	4 &\eif 2\nmid t_1t_2,\\
	8 &\eif 2\mid t_1 t_2 \eand 2\nmid (t_1, t_2),\\
	\frac{33}{2} &\eif  4\mid (t_1,t_2) \eand t_1\not \equiv t_2 \imod{16},\\
		\frac{35}{2} &\eif  4\mid (t_1,t_2) \eand t_1\equiv t_2\imod{16}.
	\end{cases}
	\end{equation*}
	Furthermore,
	\begin{equation*}	
	\frac{\Stwo}{2^{5k-5}}=\begin{cases}
	4 &\emph{\text{ if }}   2\nmid t,\\
	\frac{35}{2} &\emph{\text{ if }}  4\mid t.
	\end{cases}
	\end{equation*}
\end{lem}

\begin{proof}

If $2\nmid t_1t_2$ then by the definition of $S(t_1, t_2; 2^k)$, Theorem \ref{thrm: ch-thrm-five}, and the fact that there are $\varphi(2^k)$ invertible  elements in $\ZZ/2^k\ZZ$, we have that
$$S(t_1,t_2;2^k)=\varphi(2^k)\cdot(2^{2k-1})^2=2^{5k-3}.$$

If $2\mid t_1 t_2$ and $2\nmid (t_1, t_2)$, we consider two cases.

\noindent{Case 1:} First suppose $2\nmid t_1$, $2\mid t_2$, and $4\nmid t_2$.  Then from Theorem \ref{thrm: ch-thrm-five} and 
Lemma \ref{lemma: teven} we have, for $k\geq 3$, that
\begin{align*}
		S(t_1, t_2, 2^k)=&~2^{2k-1}\left( (2^{2k}+2^{2k-1}-2^{\frac{3k}2+\frac{1-(-1)^k}4-1})+
		\sum_{1\leq i< \frac {k+1}2} \varphi(2^{k-2i+1})\left(2^{2k}+2^{2k-1}-3\cdot 2^{2k-i-1}\right)\right.\nonumber\\
		&+\tfrac{(1-(-1)^k)}2\left(2^{2k}+2^{2k-1}- 2^{\frac{3k-1}2}\right)+\tfrac{(1+(-1)^k)}2\bigg(\left(2^{2k}+2^{2k-1}- 2^{\frac{3k}2-1}\right)+\left(2^{2k}+2^{2k-1}-3\cdot 2^{\frac{3k}2-1}\right)\bigg)\nonumber \\
		&+\sum_{2\leq i <\frac{k}2}\varphi(2^{k-2i+1})\left(2^{2k}+2^{2k-1}-3\cdot 2^{2k-i-1}\right)\nonumber\\
		&\left.+\sum_{2\leq i <\frac{k}2}\varphi(2^{k-2i})\bigg(\left(2^{2k}+2^{2k-1}\right)+\left(2^{2k}+2^{2k-1}-2^{2k-i}\right) \bigg)\right)=2^{5k-2}.
	\end{align*}

\noindent{Case 2:}  Next suppose that $2\nmid t_1$ and $4\mid t_2$. Then from Theorem \ref{thrm: ch-thrm-five} and Lemma \ref{lemma: toddorfour} we have, for $k\geq 3$, that
\begin{align*}
S(t_1,t_2;2^k)&=2^{2k-1}\left(\sum_{\substack{u\in \zstartwo \\ r\equiv 3 \imod{4}}}\left(2^{2k}+2^{2k-1}-3\cdot 2^{2k-2}\right)+ \sum_{\substack{u\in \zstartwo \\ r\equiv 5 \imod{8}}}\left(2^{2k}+2^{2k-1}\right)
+\sum_{\substack{u\in \zstartwo \\ r\equiv 1 \imod{8}}}2^{2k}\right)\\
&=2^{2k-1}\left({\varphi(2^{k-1})}\left(2^{2k}+2^{2k-1}-3\cdot 2^{2k-2}\right)+{\varphi(2^{k-2})}\left(2^{2k}+2^{2k-1}\right)+{\varphi(2^{k-2})}2^{2k}\right)\\
&=2^{5k-2}.
\end{align*}

Next we assume $4\mid (t_1,t_2)$. If $t_1\equiv t_2 \imod{16}$ then $t_1^2\equiv t_2^2 \imod{32}$ and thus $$r:=\frac{D(t_1,u)}{4}\equiv \frac{D(t_2,u)}{4}\imod{8}.$$ Then from Theorem \ref{thrm: ch-thrm-five} and Lemma \ref{lemma: toddorfour} we have, for $k\geq 3$, that
\begin{align*}
S(t_1,t_2;2^k)&=\sum_{\substack{u\in \zstartwo \\ r\equiv 3 \imod{4}}}\left(2^{2k}+2^{2k-1}-3\cdot 2^{2k-2}\right)^2+ \sum_{\substack{u\in \zstartwo \\ r\equiv 5 \imod{8}}}\left(2^{2k}+2^{2k-1}\right)^2
+\sum_{\substack{u\in \zstartwo \\ r\equiv 1 \imod{8}}}\left(2^{2k}\right)^2\\
&={\varphi(2^{k-1})}\left(2^{2k}+2^{2k-1}-3\cdot 2^{2k-2}\right)^2+{\varphi(2^{k-2})}\left(2^{2k}+2^{2k-1}\right)^2+{\varphi(2^{k-2})}\left(2^{2k}\right)^2\\
&=2^{5k-6}\cdot 35.
\end{align*}

Similarly, if $t_1\not \equiv t_2 \imod{16}$ then $\tfrac{D(t_1,u)}4 \not \equiv \tfrac{D(t_2,u)}4 \imod{8}$ and we have that
\begin{align*}
S(t_1,t_2;2^k)&={\varphi(2^{k-1})}\left(2^{2k}+2^{2k-1}-3\cdot 2^{2k-2}\right)^2+2{\varphi(2^{k-2})}\left(2^{2k}+2^{2k-1}\right)(2^{2k})\\
&=2^{5k-6}\cdot 33.
\end{align*}
The formulas for $S(t; 2^k)$ follow by considering $t_1=t_2$ in the above formulas.
\end{proof}
We can now prove the main results of this section.
\begin{proof}[Proof of Theorem \ref{samet}]
	The result follows from Lemmas \ref{lemma: sltcoprime}, \ref{lemma: ldivides t}, \ref{lemma: stwobad}, and \ref{lemma: sltk results}.
\end{proof}

\begin{proof}[Proof of Proposition \ref{proved-cases}]
	The result follows from Lemmas \ref{lemma: ldivides t}, \ref{lemma: ldivides t_1}, and \ref{lemma: sltk results}.
\end{proof}

\section{The $\ell$ factors as the average value of  $f_\ell(t_1, p)f_\ell(t_2, p)$}\label{section: proof of prop}
For a positive integer $k$ and primes $\ell$ and $p$, we set
\begin{equation*}
\label{def-tilnulk}
f_\ell^{(k)}(t, p):=\frac{\mtp}
{\ell^{2k-2} (\ell^2-1)}. 
\end{equation*}
Then by \eqref{equation: nuell} we have $$f_\ell(t, p)=\lim_{k\rightarrow\infty} f_\ell^{(k)}(t,p).$$

We now prove that the $\ell$ factor of the universal constant is the average value of $f_\ell(t_1, p)f_\ell(t_2, p)$ as $p$ varies over primes.
\begin{proof}[Proof of Proposition \ref{prop: primesum}]
We write the proof for $\ell$ odd. The proof can be adjusted for $\ell=2$. 
From the definitions of $f_\ell$ and $f_\ell^{(k)}$ and Theorem \ref{thrm: ch-thrm-five}  we have that $f_\ell(t_1, p)f_\ell(t_2, p)=f_\ell^{(k)}(t_1, p) f_\ell^{(k)}(t_2, p)$ if $t_1^2-4p \not\equiv 0 \imod{\ell^{k}}$ and $t_2^2-4p \not\equiv 0 \imod{\ell^{k}}$. Otherwise, either $t_1^2-4p \equiv 0 \imod{\ell^{k}}$ or $t_2^2-4p \equiv 0 \imod{\ell^{k}}$. In these cases we know that the difference between $f_\ell$ and $f_\ell^{(k)}$ is $\ell^{2-\frac{k}{2}-\frac{3+(-1)^k}{4}}/(\ell^2-1)$ from Theorem \ref{thrm: ch-thrm-five},
which is bounded by a constant depending only on $\ell$. Also from \eqref{fell} we know that $f_\ell$ is bounded by a constant depending only on $\ell$. From these observations, we conclude that 
\begin{equation}
\label{difference}
\left|\sum_{p\leq x} f_\ell(t_1, p) f_\ell(t_2, p)-\sum_{p\leq x}f_\ell^{(k)}(t_1, p) f_\ell^{(k)}(t_2, p) \right| \leq C_\ell \left(\sum_{\substack{{p\leq x}\\{p\equiv 4^{*}t_1^2 \imod{\ell^k}}}} 1+\sum_{\substack{{p\leq x}\\{p\equiv 4^* t_2^2 \imod{\ell^{k}}}}}1 \right),
\end{equation}
where $C_{\ell}$ is a constant that depends only on $\ell$. Here, we denoted the multiplicative inverse of $4$ in $(\ZZ/\ell^k\ZZ)^*$ by $4^*$. From \eqref{difference} and Dirichlet's theorem on primes in arithmetic progressions, we have that 
\begin{equation}
\label{last}
\limsup_{x\rightarrow \infty} \frac{1}{\pi(x)} \sum_{p\leq x}f_\ell(t_1, p)f_\ell(t_2, p)\leq  
\limsup_{x\rightarrow \infty} \frac{1}{\pi(x)} \sum_{p\leq x}f_\ell^{(k)}(t_1, p)f_\ell^{(k)}(t_2, p)+\frac{2C_\ell}{\varphi(\ell^{k})}.
\end{equation}
Next by breaking up the sum over primes into sums over primes in distinct invertible residue classes modulo $\ell^k$ and invoking Dirichlet's theorem, we have that
$$ \lim_{x\rightarrow \infty} \frac{1}{\pi(x)} \sum_{p\leq x} f_\ell^{(k)}(t_1, p) f_\ell^{(k)}(t_2, p)=\frac{1}{\varphi(\ell^k)} \sum_{u\in (\ZZ/\ell^k \ZZ)^*} f_\ell^{(k)}(t_1, u) f_\ell^{(k)}(t_2, u),$$
where $f_\ell^{(k)} (t_i, u)=m(t_i, u;\ell^k)/(\ell^{2k-2}(\ell^2-1))$ for $i=1, 2$.
By employing this identity and formulas \eqref{def-st} and \eqref{def-tilnulk}, the inequality in  \eqref{last} can be re-written as 
$$\limsup_{x\rightarrow \infty} \frac{1}{\pi(x)} \sum_{p\leq x} f_\ell(t_1, p)f_\ell(t_2, p)\leq  
\frac{1}{(\ell-1)^3(\ell+1)^2}\frac{\Sts}{\ell^{5k-5}}+\frac{2C_\ell}{\varphi(\ell^{k})}.$$
Sending $k\rightarrow \infty$ in the right-hand side of the above inequality and applying Theorem \ref{thrm: ctonetwo} yields
\begin{equation*}
\limsup_{x\rightarrow \infty} \frac{1}{\pi(x)} \sum_{p\leq x} f_\ell(t_1, p)f_\ell(t_2, p) \leq \frac{1}{(\ell-1)^3(\ell+1)^2} \lim_{k\rightarrow \infty} \frac{\Sts}{\ell^{5k-5}}.
\end{equation*}
Similarly, we can show that
\begin{equation*}
\frac{1}{(\ell-1)^3(\ell+1)^2} \lim_{k\rightarrow \infty} \frac{\Sts}{\ell^{5k-5}} \leq
\liminf_{x\rightarrow \infty} \frac{1}{\pi(x)} \sum_{p\leq x} f_\ell(t_1, p)f_\ell(t_2, p)
\end{equation*}
to complete the proof.
\end{proof}

\section{An average result}\label{section: DKS}
In this section we 
show the existence of a positive constant $c_{t_1,t_2}$ such that $$\underset{p\leq x}{\sum \nolimits^{\prime}}\frac{H(t_1^2-4 p)H(t_2^2-4 p)}{p^2}\sim c_{t_1,t_2}\log\log x,$$ as $x\rightarrow\infty$. To do this we apply a result of David, Koukoulopoulos, and Smith \cite[Theorem 4.2]{DKS}. To state their theorem in a way applicable to our problem, we need the following notations and conditions specialized to our case.

For integers $t_1$, $t_2$ and $x>\max\{3, t_1^2/4, t_2^2/4\}$, let 
$D(t_i, x)=t_i^2-4x$ for $i=1, 2$. 
For primes $\ell$, $p$, and positive integer $k$ we consider the functions $\delta_\ell$, $\Delta_{\ell^k}$, $\Delta_\ell$, and $w_p$ satisfying the following conditions:

(I) The function $\delta_\ell(p)$ is defined on primes such that $\delta_\ell(p)=0$ if $\ell=p$.

(II) We have $\delta_\ell(p)\ll \frac{1}{\ell}$ for all $p \in (x,2x]$ and all primes $\ell$.

(III) For all primes $x<p\leq 2x$ such that $\ell \nmid D(t_i,p)$, there exist complex numbers $\lambda_1(p)$ and $\lambda_2(p)$, bounded as functions of $p$, and $\eta>0$ such that  $$\delta_\ell(p)=\frac{\lambda_1(p)}{\ell}\left(\frac{D(t_1,p)}{\ell}\right)+\frac{\lambda_2(p)}{\ell}\left(\frac{D(t_2,p)}{\ell}\right) +O\left(\frac{1}{\ell^{1+\eta}}\right).$$

(IV) For every prime $\ell$ and every $k\geq 1$ the function $\Delta_{\ell^k} : \ZZ\rightarrow \CC$ is such that:

\indent\indent (a) $\Delta_{\ell^k}$ is $\ell^k$-periodic.

\indent\indent (b) $\delta_\ell(p)=\Delta_{\ell^k}(p)$ if $p\imod{\ell^k} \in \{u \in \zstar : (t_1^2-4u)(t_2^2-4u)\not \equiv 0 \imod{\ell^k}\}.$

\indent \indent(c) $\Delta_{\ell^k}$ has a mean value as $k \rightarrow \infty$ over $\zstar$, that is, there exists a $\Delta_\ell\in \CC$ such that $$\lim_{k\rightarrow \infty}\frac{1}{\varphi(\ell^k)}\sum_{u \in \zstar} \Delta_{\ell^k}(u)=\Delta_\ell.$$
\indent \indent Moreover, $|1+\Delta_\ell|\gg 1.$

\indent\indent (d) $\parallel \Delta_{\ell^k}\parallel_\infty =\sup_{u\in \mathbb{Z}} |\Delta_{\ell^k}(u)| \ll \frac{1}{\ell},$ for all $k\geq 1$.

(V)  For $q\leq Q:=\exp\{(\log\log x)^2\}$ and $u\in (\ZZ/q\ZZ)^*$ there exists $\widetilde{W}>0$ such that
\begin{equation*} \label{eqn: wtilde}
\sum_{\substack{x<p\leq 2x \\ p \equiv u \imod{q}}}|w_p| \ll \frac{\widetilde{W}}{\varphi(q)},
\end{equation*}
where $w_p$ is a function defined on primes $p\in (x,2x]$.

The following is Theorem 4.2 of \cite{DKS} written under the special conditions described above.

\begin{thrm}[\textbf{David-Koukoulopoulos-Smith}]\label{thrm: DKS}
With the above notation, assume conditions (I), (II), (III), (IV), and (V) hold and let $\epsilon>0$ and $C\geq 1$ be fixed. 
Let $$P_p:=\prod_{\ell}(1+\delta_\ell(p)) \qquad\text{ and } \qquad   W:=\sum_{x<p\leq 2x} w_p       .$$
Then the infinite product $$P:=\prod_{\ell}(1+\Delta_\ell)$$ converges absolutely and we have that 
\begin{equation}
\label{last-equation}
\sum_{x<p\leq 2x}w_pP_p = P\cdot \Bigg(W+O_{\eta, \epsilon,C}\Bigg(\frac{\widetilde{W}}{(\log x)^C}+Mx^\epsilon+(\log\log x)^{O(1)}\widetilde{W}^{\frac{2}{3}}E^{\frac{1}{3}}\Bigg)\Bigg),
\end{equation}
where $$M=\max_{\substack{1\leq i \leq 2 \\ n\neq 0}}\sum_{\substack{x<p\leq 2x \\ D(t_i,p)/n \emph{\text{ is a square}}}}|w_p| \qquad\text{ and } \qquad	E=\sum_{q\leq Q}\max_{u\in(\ZZ/q\ZZ)^*}\Bigg|\sum_{\substack{x<p\leq 2x \\ p\equiv u \imod{q}}}w_p-\frac{W}{\varphi(q)}\Bigg|.$$
\end{thrm}

We are now ready to prove the following average result. 

\begin{thrm}\label{thrm: lang-trotter average}
Let $t_1,t_2\in \ZZ$ and $x\geq \max\{3, t_1^2/4, t_2^2/4\}$. Then there exists an explicit constant $c_{t_1,t_2}>0$ such that
$$\underset{p\leq x}{\sum \nolimits^{\prime}} \frac{H(t_1^2-4 p)H(t_2^2-4 p)}{p^2}=c_{t_1,t_2}\int_2^x\frac{du}{u\log u}+O_{t_1,t_2}\left(1\right).$$
Furthermore, if $t=t_1=\pm t_2$ then  $c_{t,\pm t}$ is given by $\eqref{ctequals}$. If $t_1\neq \pm t_2$ then $c_{t_1,t_2}$ is the constant given in Theorem \ref{thrm: ctonetwo}.
\end{thrm}

\begin{proof} 
By splitting up the sum into dyadic intervals we see that it suffices to show that
	\begin{equation*}
\underset{x<p\leq 2x}{\sum \nolimits^{\prime}} \frac{H(t_1^2-4 p)H(t_2^2-4 p)}{p^2}=c_{t_1,t_2}\int_x^{2x}\frac{du}{u\log u}+O_{t_1,t_2}\left(\frac{1}{\log x}\right).
	\end{equation*}
We then, by Theorem \ref{theorem: product formula}, observe that 
$$\frac{H(t_1^2-4 p)H(t_2^2-4 p)}{p^2}
	=w_pP_p=w_p\cdot \prod_{\ell}(1+\delta_\ell(p)),$$
where
$$w_p={f_\infty(t_1,p)f_\infty(t_2,p)f_p(t_1,p)f_p(t_2,p)}$$
and
$$\delta_\ell(p):=\begin{cases} f_\ell(t_1,p)f_\ell(t_2,p)-1 &\text{ if }	\ell\neq p, \\ 0 & \text{ otherwise}.\end{cases}$$

We now check that the conditions (I), (II), and (III), assumed in Theorem \ref{thrm: DKS}, hold for $\delta_\ell(p)$. 
We have that (I) holds by the definition of $\delta_\ell(p)$ and (II) also holds, since from \eqref{fell} we have that $$\delta_\ell(p)=
\left(1+O\left( \frac{1}{\ell} \right) \right)^2-1
=O\left(\frac{1}{\ell}\right).$$
Also, from \eqref{fell} we have, for $\ell \nmid D(t,p)$, that $$f_\ell(t,p)=1+\frac{\left(\tfrac{D(t,p)}{\ell}\right)}{\ell}+O\Bigg(\frac{1}{\ell^2}\Bigg).$$
Thus, (III) holds with $\lambda_1=\lambda_2=1$.

Next, for $u\in \mathbb{Z}$, we define
\begin{equation*}
\Delta_{\ell^k}(u):=
-1+\frac{ \mtone\mttwo}{\ell^{4k-4}(\ell^2-1)^2}.
\end{equation*}
By employing Theorem \ref{thrm: ch-thrm-five} we can show that (IV.a), (IV.b), and (IV.d) hold for $\Delta_{\ell^k}(u)$. We also have that (IV.c) holds with
\begin{equation*}
\Delta_\ell=-1+\frac{1}{(\ell-1)^3(\ell+1)^2}\lim_{k\rightarrow \infty}\frac{\Sts}{\ell^{5k-5}}
\end{equation*}
since the limit exists by Theorem \ref{thrm: ctonetwo} and by Theorem \ref{thrm: ch-thrm-five} $\lim_{\ell\rightarrow \infty} (1+\Delta_\ell)=1$.

Lastly, by the Brun-Titchmarsh inequality,  (V) holds for $w_p$ as defined above with $\widetilde{W}=\log\log x$.

Thus we may apply Theorem \ref{thrm: DKS} with $\epsilon>0$ and $C>1$. It remains to show that the error terms are not too large. We first note that by Theorem \ref{thrm: ch-thrm-five} and the definition of $f_\infty$ we have
\begin{align*}
W=& \sum_{x<p\le 2x} w_p = \frac{1}{\pi^2}\sum_{x<p\leq 2x}\left(\frac{1}{p}+O_{t_1,t_2}\left(\frac{1}{p^{2}}\right)\right)=\frac{1}{\pi^2}\int_x^{2x}\frac{du}{u\log u}+O_{t_1,t_2}\left(\frac{1}{\log x}\right).
\end{align*}
We also find that 
\begin{equation*}
E\ll\sum_{q\leq Q}\max_{(u,q)=1}\Bigg|\sum_{\substack{x<p\leq 2x \\ p\equiv u \imod{q}}}\frac{1}{ p}-\frac{1}{\varphi(q)}\sum_{x<p\leq 2x}\frac{1}{p}\Bigg|\ll \frac{1}{(\log x)^B},
\end{equation*}
for any fixed $B>0$, by the Bombieri-Vinogradov theorem. Furthermore, we have that
$$M=\max_{\substack{n\leq -1 \\ 1\leq i \leq 2}} \sum_{\substack{x<p\leq 2x \\ (t_i^2-4p)/n \text{ is a square}}}w_p \ll \max_{\substack{n\leq -1 \\ 1\leq i \leq 2}} \frac{\#\{m\in \ZZ: 4x<|n|m^2+t_i^2\leq 8x\}}{x}\ll \frac{1}{\sqrt{x}}.$$
Thus, we conclude that in \eqref{last-equation} the error terms involving $\widetilde{W}$, $E$, and $M$ are smaller than the main term of $W$ and the result follows.
\end{proof}

\section{Computational evidence for Conjecture \ref{conj: constant-distinct-ts}}\label{section: comp}
In this section we consider $c_{t_1,t_2}$ for the cases not covered by Theorem \ref{samet} and Proposition \ref{proved-cases}. We have written a program in SAGE to compute $\Sk=\Sts/\ell^{5k-5}$, using Theorem \ref{thrm: ch-thrm-five}, for the first few small primes $\ell$ and various integers $t_1,t_2$, and $k$. We then used rational interpolation approximation in MAPLE to produce a formula for $\Sk$ as a rational function in terms of $\ell$. We summarize our findings below.

\begin{conj}\label{differentt} Let  $t_1$ and $t_2$ be two fixed integers, where $t_1 \neq \pm t_2$. Let $\ell$ be a prime and let $$\alpha=\alpha(t_1,t_2,\ell):=\max\{\nu_{\ell}(t_1+t_2),\nu_{\ell}(t_1-t_2)\}.$$ 

If $\ell \nmid  2t_1 t_2$ and  $k\geq \alpha+1$, then
	\begin{equation*}
\frac{\Sts}{\ell^{5k-5}}=\begin{cases} \ell^2(\ell^3-\ell^2-\ell-2) -\ell^3& \eif \alpha=0, \\
	\ell^2(\ell^3-\ell^2-\ell-2)+\frac{\ell^2(\ell^{2\alpha}-\ell^2-\ell-1)}{\ell^{2\alpha}(\ell+1)} & \eif \alpha\geq 1.
	\end{cases}
	\end{equation*}

If $\ell=2$ and $k\geq \alpha+1$, then 
\begin{equation*}
\frac{\Sttwo}{2^{5k-5}}=\begin{cases}
15 & \eif 2\mid (t_1,t_2), 4\nmid (t_1,t_2),~{\rm and}~ \alpha=1,\\
\frac{103}6 - \frac{7}{3\cdot2^{2\alpha-3}}& \eif 2\mid (t_1,t_2), 4\nmid(t_1,t_2),~{\rm and}~\alpha\geq 3.
\end{cases}
\end{equation*}
\end{conj}

\begin{remark}
{\em (i) Note that in the cases in the above conjecture $\Sk$ stabilizes at $k=\alpha+1$. }	

{\em (ii) We have verified the truth of Conjecture \ref{differentt} for all integer pairs $(t_1, t_2)$ with $1\leq t_1, t_2\leq 100$ for which $\alpha=0$, for primes $2, 3, \ldots, 17$ (the first 7 primes), and for $1\leq k\leq 4$.}

{\em (iii) Note that if we can show that, for odd primes $\ell$ and $\alpha=0$, $\Sk$ is a polynomial of degree $5$, then our computations in (ii) provide a proof of the conjecture for odd primes $\ell$ in this case.}

{\em (iv) We have also verified the truth of conjecture \ref{differentt} for all integer pairs $(t_1, t_2)$ with $1\leq t_1, t_2\leq 100$ for which $0\leq \alpha \leq 5$, for primes $2, 3, \ldots,19$ (the first 8 primes), and for $\alpha+1\leq k\leq \alpha+3$.} 

\end{remark}

\subsection*{Acknowledgements}
The authors would like to thank Julia Gordon for correspondence on an earlier version of this paper and for her clarifying comments regarding the stability of $\mathcal{S}_k$. We also would like to thank Jeff Achter and Jesse Thorner for correspondence and comments. The first author thanks Forrest Francis for help with computations in Lemma \ref{lemma: ldivides t_1}.

\nocite{*}

\begin{rezabib}

\bib{ADJ}{article}{
   author={Akbary, Amir},
   author={David, Chantal},
   author={Juricevic, Robert},
   title={Average distributions and products of special values of
   $L$-series},
   journal={Acta Arith.},
   volume={111},
   date={2004},
   number={3},
   pages={239--268},
   issn={0065-1036},
   review={\MR{2039225 (2004k:11086)}},
   doi={10.4064/aa111-3-3},
}

\bib{BJ}{article}{
   author={Baier, Stephan},
   author={Jones, Nathan},
   title={A refined version of the Lang-Trotter conjecture},
   journal={Int. Math. Res. Not. IMRN},
   date={2009},
   number={3},
   pages={433--461},
   issn={1073-7928},
   review={\MR{2482121}},
   doi={10.1093/imrn/rnn136},
}

\bib{CH}{article}{
	author={Wouter Castryck},
	author={Hendrik Hubrechts},
	title={The distribution of the number of points modulo an integer on elliptic curves over finite fields},
	journal={Ramanujan J.},
	volume={30},
	date={2013},
	number={2},
	pages={223--242},
	issn={},
	review={\MR{3017927}},
	doi={},
} 

\bib{CDKS}{article}{
   author={Chandee, Vorrapan},
   author={David, Chantal},
   author={Koukoulopoulos, Dimitris},
   author={Smith, Ethan},
   title={The frequency of elliptic curve groups over prime finite fields},
   journal={Canad. J. Math.},
   volume={68},
   date={2016},
   number={4},
   pages={721--761},
   issn={0008-414X},
   review={\MR{3518992}},
   doi={10.4153/CJM-2015-013-1},
}

\bib{CDSS}{article}{
	author={Cojocaru, Alina},
	author={Davis, Rachel},
	author={Silverberg, Alice},
	author={Stange, Katherine}
	title={Arithmetic properties of the Frobeniius traces defined by a rational abelian variety, with two appendices by J-P. Serre},
	journal={Int. Math. Res. Not. IMRN},
	date={2017},
	number={12},
	pages={3557--3602},
	issn={},
	review={},
	doi={10.1093/imrn/rnw058},
}

\bib{DKS}{article}{
	author={David, Chantal},
	author={Koukoulopoulos, Dimitris},
	author={Smith, Ethan},
	title={Sums of Euler products and statistics of elliptic curves},
	journal={Math. Ann.},
	date={2017},
	volume={368}
	number={1--2},
	pages={685--752},
	issn={},
	review={},
	doi={},
}

\bib{DP}{article}{
	author={David, Chantal},
	author={Pappalardi, Francesco},
	title={Average Frobenius distributions of elliptic curves},
   journal={Int. Math. Res. Not.},
   date={1999},
   number={4},
   pages={165--183},
   issn={},
   review={},
  doi={},
}

\bib{Elk}{article}{
	author={Elkies,  Noam},	
	title={Distribution  of  supersingular  primes. Journ\'{e}es Arithm\'{e}tiques, 1989 (Luminy, 1989)},
	journal={Ast\'{e}risque No. 198--200 (1991)},
	date={1992},
	pages={127--132},
	review={\MR{1144318}}
}

\bib{F}{article}{
	author={Gerd Faltings},
	title={Endlichkeitssa\"tze fu\"r abelsche Varieta\"ten u\"ber Zahlko\"rpern},
	journal={Invent. Math.},
	volume={73},
	year={1983},
	number={3},
	pages={349-366},
	review={\MR{0718935}},
	 doi={10.1007/BF01388432},
}

\bib{FM}{article}{
	   author={Fouvry, {\'E}.},
	   author={Murty, M. Ram},
	title={On the distribution of supersingular primes},
	journal={Canad. J. Math.},
	volume={48},
	year={1996},
	number={1}, 
	pages={81--104},
}

\bib{FM2}{article}{
   author={Fouvry, {\'E}.},
   author={Murty, M. Ram},
   title={Supersingular primes common to two elliptic curves},
   conference={
      title={Number theory},
      address={Paris},
      date={1992--1993},
   },
   book={
      series={London Math. Soc. Lecture Note Ser.},
      volume={215},
      publisher={Cambridge Univ. Press, Cambridge},
   },
   date={1995},
   pages={91--102},
   review={\MR{1345175}},
   doi={10.1017/CBO9780511661990.007},
}

\bib{G}{article}{
   author={Gekeler, Ernst-Ulrich},
   title={Frobenius distributions of elliptic curves over finite prime
   fields},
   journal={Int. Math. Res. Not.},
   date={2003},
   number={37},
   pages={1999--2018},
   issn={1073-7928},
   review={\MR{1995144 (2004d:11048)}},
   doi={10.1155/S1073792803211272},
}

\bib{HSBT}{article}{
	author={Harris, Michael},
	author={Shepherd-Barron, Nicholas},
	author={Taylor, Richard},
	title={A family of Calabi-Yau varieties and potential automorphy.},
	journal={Ann. of Math.},
	volume={171},
	date={2010},
	number={2},
	pages={779--813},
	issn={},
	review={},
	doi={},
}

\bib{IR}{book}{
   author={Ireland, Kenneth},
   author={Rosen, Michael},
   title={A classical introduction to modern number theory},
   series={Graduate Texts in Mathematics},
   volume={84},
   edition={2},
   publisher={Springer-Verlag, New York},
   date={1990},
   pages={xiv+389},
   isbn={0-387-97329-X},
   review={\MR{1070716}},
   doi={10.1007/978-1-4757-2103-4},
}

\bib{Ito}{book}{
   author={It\=o, Kiyosi},
   title={Introduction to probability theory},
   note={Translated from the Japanese by the author},
   publisher={Cambridge University Press, Cambridge},
   date={1984},
   pages={x+213},
   isbn={0-521-26418-9},
   isbn={0-521-26960-1},
   review={\MR{777504}},
}

\bib{J2}{article}{
   author={Jones, Nathan},
   title={Pairs of elliptic curves with maximal Galois representations},
   journal={J. Number Theory},
   volume={133},
   date={2013},
   number={10},
   pages={3381--3393},
   issn={0022-314X},
   review={\MR{3071819}},
   doi={10.1016/j.jnt.2013.03.002},
}

\bib{K}{article}{
   author={Katz, Nicholas M.},
   title={Lang-Trotter revisited},
   journal={Bull. Amer. Math. Soc. (N.S.)},
   volume={46},
   date={2009},
   number={3},
   pages={413--457},
   issn={0273-0979},
   review={\MR{2507277}},
   doi={10.1090/S0273-0979-09-01257-9},
}

\bib{LT}{book}{
   author={Lang, Serge},
   author={Trotter, Hale},
   title={Frobenius distributions in ${\rm GL}_{2}$-extensions},
   series={Lecture Notes in Mathematics, Vol. 504},
   note={Distribution of Frobenius automorphisms in ${\rm
   GL}_{2}$-extensions of the rational numbers},
   publisher={Springer-Verlag, Berlin-New York},
   date={1976},
   pages={iii+274},
   review={\MR{0568299 (58 \#27900)}},
}

\bib{MP}{article}{
   author={M. Ram Murty},
   author={Sudhir Pujahari},
   title={Distinguishing Hecke eigenforms},
   journal={Proc. Amer. Math. Soc.},
   volume={145},
   date={2017},
   pages={1899--1904},
   }

\bib{KM}{book}{
	author={Murty, V. Kumar},
	title={ Modular forms and the Chebotarev density theorem. II},
	series={Analytic number theory (Kyoto, 1996)},
	date={1997},
	pages={287--308},
	review={\MR{1694997 }},
}

\bib{O}{article}{
   author={Oesterl\'e, Joseph},
   title={R\'eduction modulo $p^{n}$\ des sous-ensembles analytiques ferm\'es
   de ${\bf Z}^{N}_{p}$},
   language={French},
   journal={Invent. Math.},
   volume={66},
   date={1982},
   number={2},
   pages={325--341},
   issn={0020-9910},
   review={\MR{656627}},
   doi={10.1007/BF01389398},
}

\bib{S}{article}{
   author={Serre, Jean-Pierre},
   title={Propri\'et\'es galoisiennes des points d'ordre fini des courbes
   elliptiques},
   language={French},
   journal={Invent. Math.},
   volume={15},
   date={1972},
   number={4},
   pages={259--331},
   issn={0020-9910},
   review={\MR{0387283}},
}

\bib{S2}{article}{
   author={Serre, Jean-Pierre},
   title={Quelques applications du th\'eor\`eme de densit\'e de Chebotarev},
   language={French},
   journal={Inst. Hautes \'Etudes Sci. Publ. Math.},
   number={54},
   date={1981},
   pages={323--401},
   issn={0073-8301},
   review={\MR{644559}},
}

\bib{MR817210}{book}{
   author={Silverman, Joseph H.},
   title={The arithmetic of elliptic curves},
   series={Graduate Texts in Mathematics},
   volume={106},
   publisher={Springer-Verlag, New York},
   date={1986},
   pages={xii+400},
   isbn={0-387-96203-4},
   review={\MR{817210}},
}

\bib{TZ}{article}{
	author={Thorner, Jesse},
	author={Zaman, Asif},
	title={A Chebotarev Variant of the Brun-Titchmarsh Theorem and
Bounds for the Lang-Trotter conjectures},
	journal={Int. Math. Res. Not.},
        date={2017},
        pages={1--37},
        doi={10.1093/imrn/rnx031},	
	}

\bib{Z}{article}{
	author={Zywina, David},
	title={Bounds for the Lang-Trotter Conjecture},
	series={SCHOLAR—a scientific celebration highlighting open lines of arithmetic research},
	note={Contemp. Math.},
	publisher={ Centre Rech. Math. Proc., Amer. Math. Soc., Providence, RI},
	date={2015},
	number={655},
	pages={235--256},
	issn={},
	review={\MR{3453123}},
	doi={},	
}

\end{rezabib}

\end{document}